\documentclass{amsart}
\title[Triply periodic zero mean curvature surfaces]{
  Embedded triply periodic zero mean curvature surfaces of mixed type in Lorentz-Minkowski 3-space}
\date{February 19, 2013}
\usepackage[dvips]{graphicx}
\usepackage{enumerate}
\usepackage{amssymb, bm}
\usepackage{eepic}
\usepackage[usenames]{color}

\usepackage{amsthm}
\theoremstyle{plain}
 
 \newtheorem{theorem}{Theorem}[section]
 \newtheorem*{theorem*}{Theorem}
 \newtheorem{introtheorem}{Theorem}

 \newtheorem*{lemma*}{Lemma}
 \newtheorem{proposition}[theorem]{Proposition}
 \newtheorem{fact}[theorem]{Fact}
 \newtheorem*{fact*}{Fact}

 \newtheorem{lemma}[theorem]{Lemma}
 
 \theoremstyle{remark}

 \newtheorem{remark}[theorem]{Remark}
 \newtheorem*{remark*}{Remark}
 \newtheorem*{problem*}{Problem}

\numberwithin{equation}{section}
\numberwithin{figure}{section}

\newcommand{\Z}{\boldsymbol{Z}}

\newcommand{\R}{\boldsymbol{R}}
\newcommand{\C}{\boldsymbol{C}}

\renewcommand{\Re}{\operatorname{Re}}
\renewcommand{\Im}{\operatorname{Im}}

\renewcommand{\phi}{\varphi}
\renewcommand{\epsilon}{\varepsilon}

\author{S.~Fujimori}
\address[Shoichi Fujimori]{%
   Department of Mathematics, Okayama University,
   Tsushima-naka, Okayama 700-8530, Japan}
\email{fujimori@math.okayama-u.ac.jp}
\author{W.~Rossman}
\address[Wayne Rossman]{%
   Department of Mathematics, Faculty of Science,
   Kobe University,
   Rokko, Kobe 657-8501, Japan
}
\email{wayne@math.kobe-u.ac.jp}
\author{M.~Umehara}
\address[Masaaki Umehara]{%
   Department of Mathematical and Computing Sciences,
   Tokyo Institute of Technology,
   2-12-1-W8-34, O-okayama, Meguro-ku,
   Tokyo 152-8552, Japan.
}
\email{umehara@is.titech.ac.jp}

\author{K.~Yamada}
\address[Kotaro Yamada]{%
   Department of Mathematics, 
   Tokyo Institute of Technology, 1-12-1-H-7, 
   O-okayama, Meguro, Tokyo 152-8551, 
   Japan
}
\email{kotaro@math.titech.ac.jp}

\author{S.-D.~Yang}
\address[Seong-Deog Yang]{%
   Department of Mathematics,
   Korea University,
   Seoul 136-701, Korea
}
\email{sdyang@korea.ac.kr}

\thanks{
Fujimori was partially supported by the Grant-in-Aid for Young Scientists (B) No. 21740052, Rossman was supported by Grant-in-Aid for Scientific Research (B) No. 20340012, Umehara by (A) No. 22244006 and Yamada by (B) No. 21340016 from Japan Society for the Promotion of Science. 
Yang was supported in part by National Research Foundation of Korea 2012-042530.
}
\begin{document}
\begin{abstract}
We construct embedded triply periodic zero mean curvature surfaces of mixed type 
in the Lorentz-Minkowski 3-space $\R^3_1$ with the same topology 
as the Schwarz D surface in the Euclidean 3-space $\R^3$. 
\end{abstract}
\maketitle

\section{Introduction}
In any robust surface theory, it is essential to have a large 
collection of interesting examples.
One of the interesting classes of surfaces to study are the zero mean curvature surfaces of mixed type
in Lorentz-Minkowski three-space $\R^3_1$, which, roughly speaking, are smooth surfaces of mixed causal type with mean curvature, wherever it is well defined, equal to zero.   

Several authors have found such examples \cite{K},
\cite{G}, \cite{ST}, \cite{Kl}, \cite{CR2},
all of which have simple topology. The main goal of this article is to provide a concrete example of a family of such surfaces with nontrivial topology.

The motivation for the method of our construction is the fact that fold singularities of spacelike maximal surfaces have  real analytical extensions to 
timelike minimal surfaces (cf. \cite{G}, \cite{Kl}, \cite{KKSY},
\cite{CR2}). 
Main ingredients are the spacelike maximal analogues in $\R^3_1$ of the Schwarz P surfaces and the Schwarz D surfaces in $\R^3$, which were remarked upon in a previous work \cite{FRUYY} by the authors.  The Schwarz P-type maximal surfaces admit
cone-like singularities while the Schwarz D-type maximal surfaces 
admit fold singularities (cf. Figure \ref{fig:surface_PD}). 
By extending the Schwarz D-type (spacelike) maximal surfaces to timelike minimal surfaces, we obtain the following main result of this article:

\begin{introtheorem}\label{th:A}
The $1$-parameter family of
Schwarz D-type spacelike maximal surfaces $\{X_a\}_{0<a<1}$
has a unique analytic extension
$$
\tilde X_a:\Sigma_a \to \R^3_1/\Gamma_a \qquad (0<a<1)
$$
to embedded zero mean curvature surfaces,
where 
$\R^3_1/\Gamma_a$ is a torus
given by a suitable 3-dimensional lattice $\Gamma_a$,
and $\Sigma_a$ is a closed orientable $2$-manifold
of genus three (cf. Figure~\ref{fig:mixface}).
\end{introtheorem}

\begin{figure}[thbp] 
\begin{center}
\begin{tabular}{c@{\hspace{3em}}c}
 \includegraphics[width=.30\linewidth]{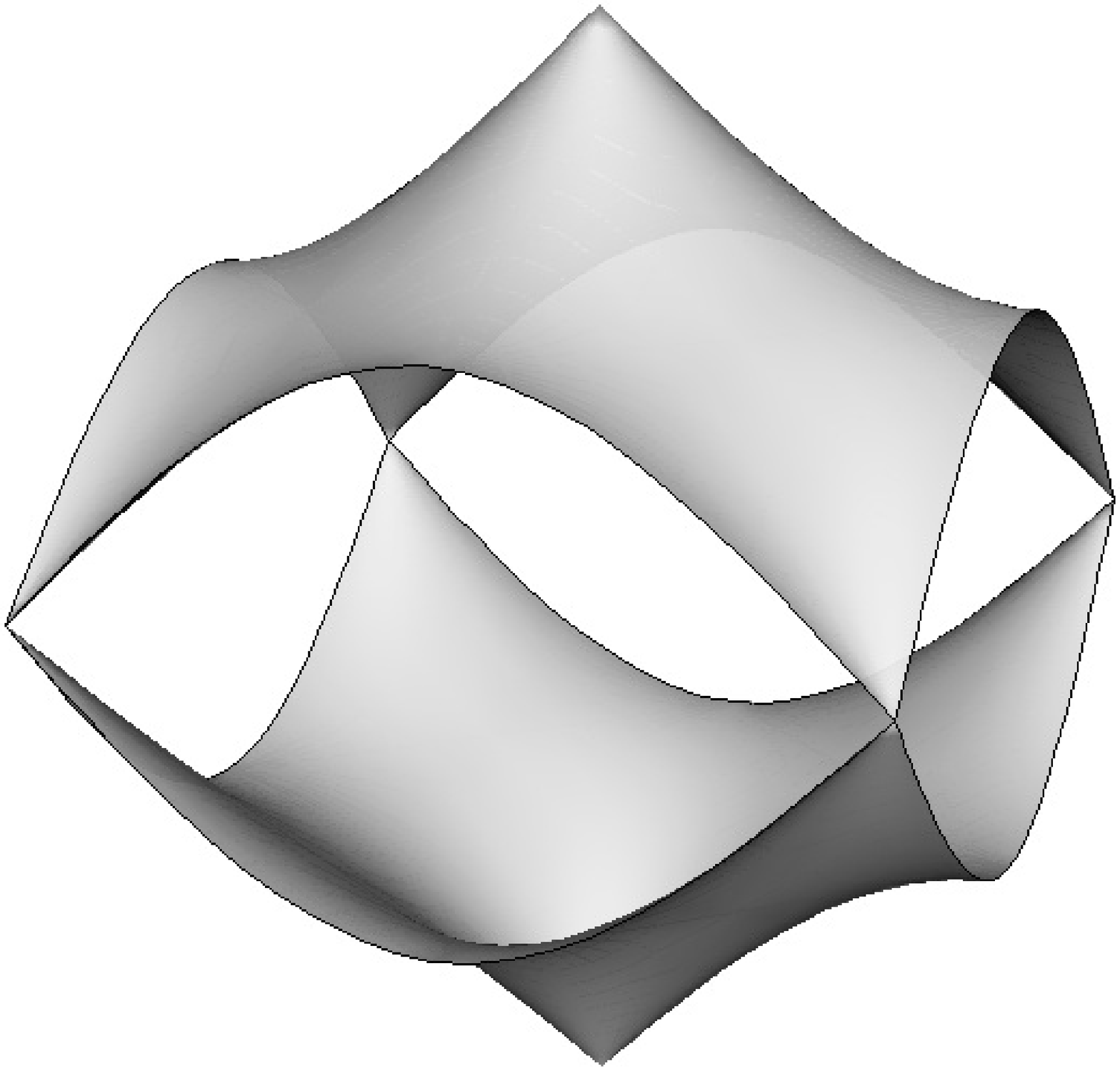} &
 \includegraphics[width=.30\linewidth]{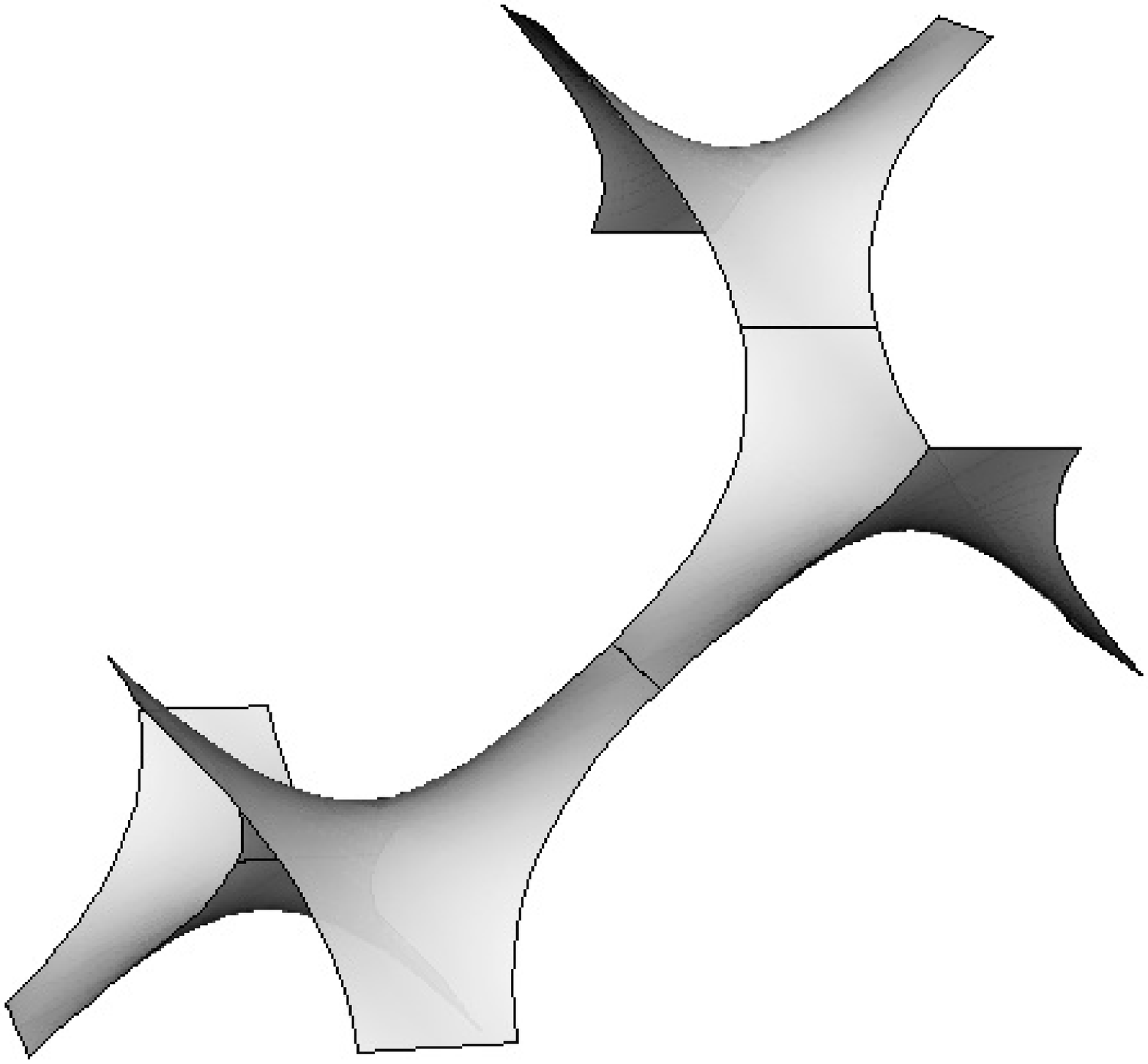} \\
\end{tabular}
\end{center}
\caption{Schwarz P-type (left) and D-type maximal surfaces (right).}
\label{fig:surface_PD}
\end{figure} 

In so doing we provide a concrete description of the family of triply periodic maximal surfaces containing the Schwarz P-type and D-type maximal surfaces.

\begin{figure}[htbp] 
\begin{center}
\begin{tabular}{ccc}
 	\scalebox{0.6}{\includegraphics[width=.54\linewidth]{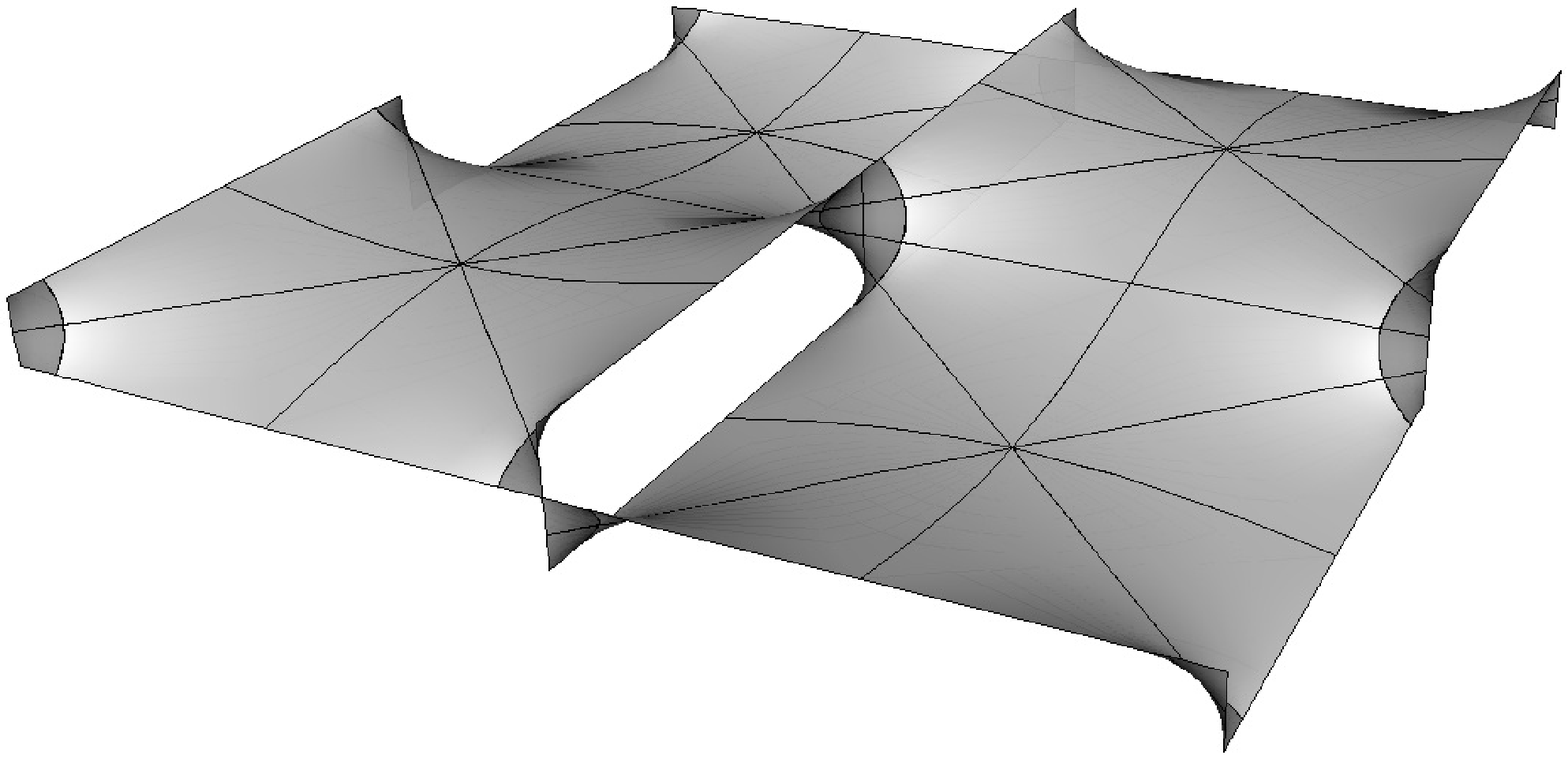}} &
	\scalebox{0.6}{\includegraphics[width=.50\linewidth]{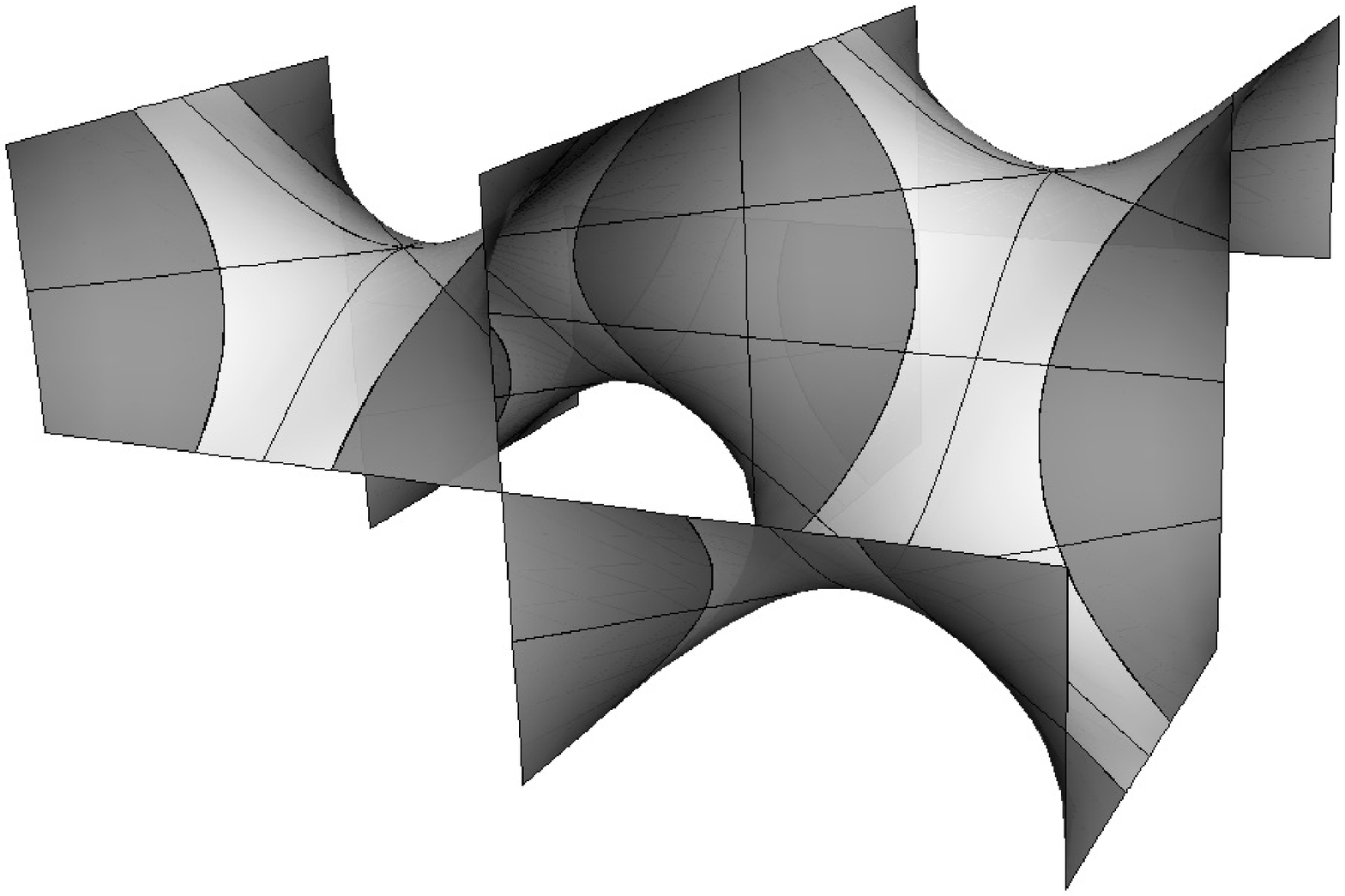}}   &
 	\scalebox{0.6}{\includegraphics[width=.40\linewidth]{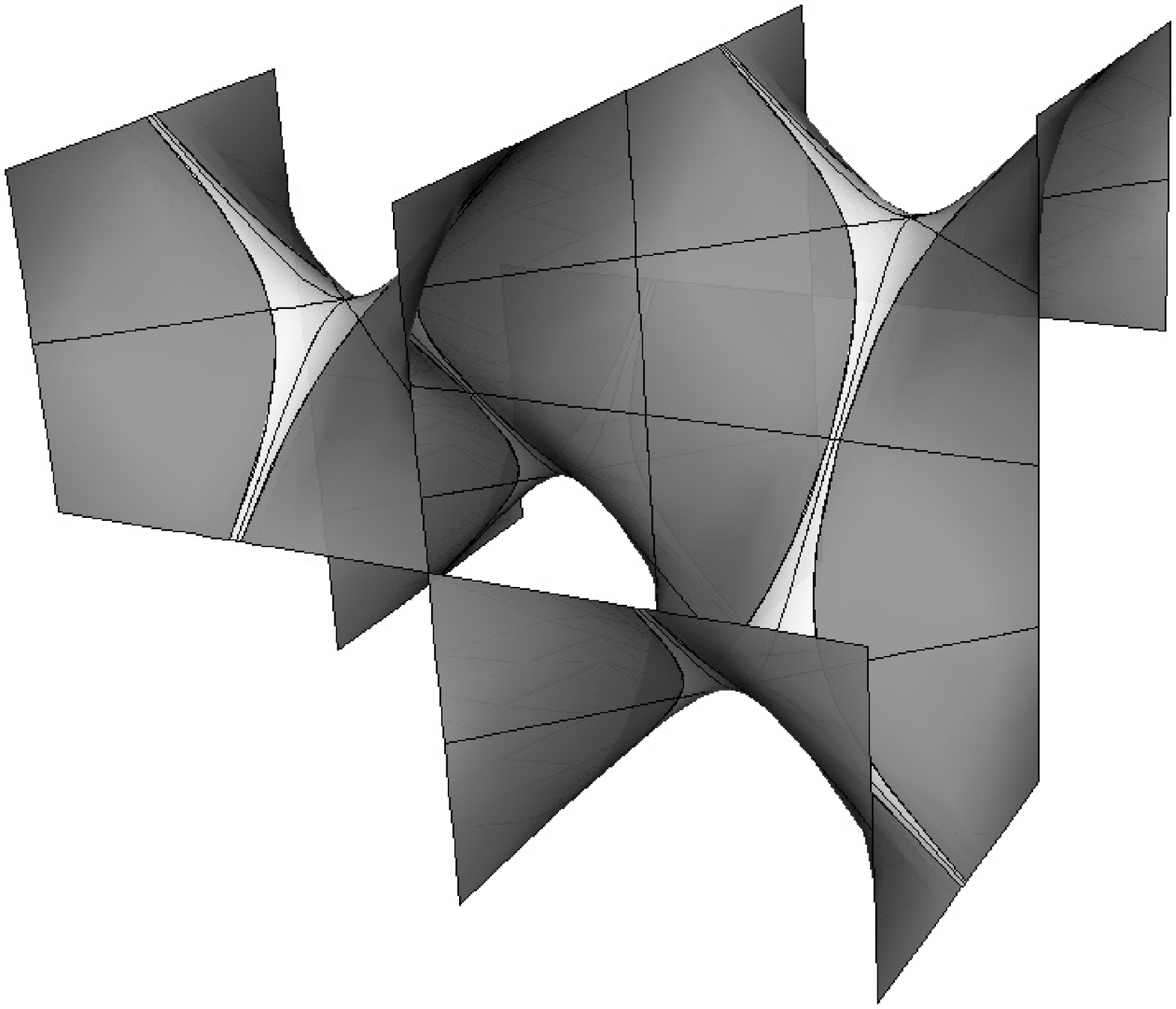}} 
\end{tabular}
\end{center}
\caption{
Embedded triply periodic zero mean curvature surfaces of mixed type constructed in this article for 
         $a=0.1$ (left), $a=(\sqrt{3}-1)/\sqrt{2}\approx 0.52$ (center), and $a=0.9$ (right). 
         The spacelike parts are indicated by grey shades 
         and the timelike parts are indicated by black shades.
}
\label{fig:mixface}
\end{figure} 

\section{Triply periodic maximal surfaces}
\label{sec:triply}

In this section, we construct triply periodic maximal 
surfaces in $\R^3_1$ based on the Schwarz 
P and D minimal surfaces in $\R^3$. 
We use either
$t,x,y$ or $x_0, x_1, x_2$ to denote the standard coordinates of $\R^3_1$. 

Take the hyperelliptic Riemann surface 
\[
   M_a  :=\left\{ 
         (z,w) \in (\C \cup \{ \infty \})^2 \, ; 
    	 \, w^2 = z^8+(a^4+a^{-4})z^4+1  \right\} 
\]
of genus $3$,
where $a\in (0,1)$ is a real constant.  
Take the Weierstrass data
\[
              G:=z,\qquad \eta_{\theta} :=e^{i\theta}\frac{dz}{w}
              \qquad \bigl(\theta\in [0,\pi),\,\, i:=\sqrt{-1}\bigr)
\]
on $M_a$, and set 
\begin{equation}\label{eq:min-r3}
   \hat{f}_{a,\theta} :=\Re\int \bigl(1-G^2,i(1+G^2),2G\bigr)\eta_\theta.
\end{equation}  
Then $\hat{f}_{a,\theta}$ gives a minimal surface in $\R^3$. 
When $a=(\sqrt{3}-1)/\sqrt{2}$, that is, when $a^4+a^{-4}=14$, 
$\hat{f}_{a,0}$ (resp. $\hat{f}_{a,\pi/2}$) is called the {\it Schwarz P surface} 
(resp. the {\it Schwarz D surface}). 
Also, for $a\in (0,1)$, 
$\hat{f}_{a,0}$ (resp. $\hat{f}_{a,\pi/2}$) is called the {\it Schwarz P family} 
(resp. the {\it Schwarz D family}). 
For the period computation for those minimal surfaces, 
we refer to \cite{R}.

Now, for the same Riemann surface $M_a$ and the Weierstrass data $(G,\eta_\theta)$ as above, 
we set 
\[
   f_{a,\theta}:=\Re\int \Phi_\theta :\widetilde{M_a}\longrightarrow\R^3_1,
\]  
where 
 \begin{equation}\label{eq:maxface}
\Phi_\theta:=
\bigl(-2G, 1+G^2,i(1-G^2)\bigr)\eta_\theta.
 \end{equation}
Then $f_{a,\theta}$ gives a {\it maxface}
(i.e. a maximal surface with admissible singularities,
see \cite{CR2})
in Lorentz-Minkowski $3$-space $\R^3_1$ of
signature $(-,+,+)$.
A point $p\in M_a$ is a singular point
if and only if $|G(p)|=1$, and
a singular point $p$ is a cuspidal  edge point
if and only if $\Im (dG/(G^2\eta)) \ne 0$ at $p$ 
(cf. \cite[Fact 1.3]{FRUYY}).
Using this, one can easily check that
$f_{a,\theta}$ admits only cuspidal edge
singularities whenever $\theta\ne 0, \pi/2$
for each $a\in (0,1)$.
On the other hand, if $\theta=0$
then $f_{a,0}$ admits only cone-like singularities
(cf. \cite[Lemma 2.3]{FRUYY}).
Later, we will show that $f_{a,0}$ is triply periodic.
Since $f_{a,0}$ has the same Weierstrass data
as the Schwarz P surface in Euclidean $3$-space,
we call $f_{a,0}$ the {\it Schwarz P-type maximal surface}.

As pointed out in \cite[Definition 2.1]{KY2} and
\cite[Proposition 2.14]{CR2},
there exists 
a duality
between fold singularities and generalized cone-like singularities
via conjugation of maximal surfaces.
Since $f_{a,\pi/2}$ is the conjugate surface
of $f_{a,0}$, we can conclude that
$f_{a,\pi/2}$ admits only fold singularities 
(cf. \cite{CR2}).  
Later, we also show that $f_{a,\pi/2}$ is triply periodic.
Since $f_{a,\pi/2}$ has the same Weierstrass data
as the Schwarz D surface in Euclidean $3$-space,
we call $f_{a,\pi/2}$ the {\it Schwarz D-type maximal surface}.

\medskip
The surface $f_{a,0}$ has the following symmetries:

\begin{lemma}\label{lm:phi1-4}
It holds that 
\[
\begin{alignedat}{2}
\phi_1^*(\Phi_0)^T &=
\begin{pmatrix} 1 & 0 & \phantom{-}0 \\
                0 &          1 & \phantom{-}0 \\
                0 & 0 & -1
\end{pmatrix} \overline{(\Phi_0)}^T,
\qquad
&\phi_2^*(\Phi_0)^T &=
\begin{pmatrix}          - 1 & \phantom{-}0 & \phantom{-}0 \\
                \phantom{-}0 &          - 1 & \phantom{-}0 \\
                \phantom{-}0 & \phantom{-}0 &          - 1
\end{pmatrix} (\Phi_0)^T, \\
\phi_3^*(\Phi_0)^T &=
\begin{pmatrix}          - 1 & \phantom{-}0 & \phantom{-}0 \\
                \phantom{-}0 & \phantom{-}0 & \phantom{-} 1 \\
                \phantom{-}0 & {-}1 & \phantom{-}0
\end{pmatrix} (\Phi_0)^T, 
\qquad
&\phi_4^*(\Phi_0)^T &=
\begin{pmatrix}          - 1 & \phantom{-}0 & 0 \\
                \phantom{-}0 &          - 1 & 0 \\
                \phantom{-}0 & \phantom{-}0 & 1
\end{pmatrix} (\Phi_0)^T, 
\end{alignedat}
\]
where $(\Phi_0)^T$ is the transpose of $\Phi_0$
and $\phi_j^*(\Phi_0)^T$ $(j=1,2,3,4)$
is the pull-back of the 
$\C^3$-valued 1-form $(\Phi_0)^T$ by
the maps  $\phi_j:M_a\to M_a$  
given by 
\[
 \begin{alignedat}{2}
  \phi_1(z,w) &:= (\bar z,\bar w),\qquad
  &\phi_2(z,w) &:= (z, -w),\\
  \phi_3(z,w) &:= (iz,w),\qquad
  &\phi_4(z,w) &:= \left(\frac{1}{z},\frac{w}{z^4}\right).
 \end{alignedat}
\]
\end{lemma}

In the following discussion, we apply
only the symmetry with respect to $\phi_3$.
Using this, we examine the period of $f_{a,\theta}$.  
We set
$$
b:=a^4+a^{-4}.
$$
We define the following four oriented 
regular arcs on $M_a$:
\begin{align*}
c_1(t)&:=\left(-i t,\sqrt{t^8+bt^4+1}\right) 				& t&\in [-\infty,0],	 	\\
c_2(t)&:=\left(t,\sqrt{t^8+bt^4+1}\right) 					& t&\in [0,+\infty], 		\\
c_3(t)&:=\left(-it,\sqrt{t^8+bt^4+1}\right) 				& t&\in [-1,1], \\
c_4(t)&:=\left(e^{it},-e^{2it}\sqrt{2\cos 4t+b}\right) 	& t&\in [-\pi /2,\pi /2],
\end{align*}
where all of the four square roots take positive real values.
We then define two oriented loops  
$\gamma_1: [-\infty,+\infty] \to M_a$ and $\gamma_2: [-2,\pi] \to M_a$ 
by 
\begin{equation}\label{eq:gamma12}
\gamma_1(s) := \begin{cases} 	c_1(s) & \text{ if } s \in [-\infty,0], \\
								c_2(s) & \text{ if } s \in [0, \infty]. \end{cases} 
\quad 
\gamma_2(s) :=\begin{cases} 	c_3(s+1) & \text{ if } s \in [-2,0], \\
								c_4(s-\pi/2) & \text{ if } s \in [0, \pi]. \end{cases} 
\end{equation}
The fundamental group $\pi_1(M_a)$ of $M_a$ is 
generated by eight loops
$$
\gamma_k,\quad \phi_3\circ \gamma_k,
\quad (\phi_3)^2\circ \gamma_k:=\phi_3\circ \phi_3\circ\gamma_k,
\quad (\phi_3)^3\circ \gamma_k:=\phi_3\circ\phi_3\circ 
\phi_3\circ\gamma_k
\qquad (k=1,2).  
$$
One can easily prove the next lemma following the computations in \cite{R}:

\begin{lemma}\label{lm:period-triply}
We have
\[
   \oint_{\gamma_1}\Phi_0 = \bigl(-q_1(a),q_2(a),q_2(a)\bigr),\qquad
   \oint_{\gamma_2}\Phi_0 = \bigl(iq_3(a),-iq_4(a),q_2(a)\bigr),
\]
where $q_j(a)$ $(j=1,2,3,4)$ are positive real
numbers given by
\begin{align*}
q_1(a) &:= \int_0^\infty
        \frac{4ds}{\sqrt{(b+2)s^4-2(b-6)s^2
+b+2}}=\int_0^1\frac{8t}{\sqrt{t^8+bt^4+1}}dt,
          \\
q_2(a) &:= \int_0^\infty
          \frac{ds}{\sqrt{s^4+s^2+(b+2)/16}}=
\int_0^1\frac{2(1+t^2)}{\sqrt{t^8+bt^4+1}}dt,
          \\
q_3(a) &:= \int_0^\infty
        \frac{4ds}{\sqrt{(b+2)s^4+2(b-6)s^2+
b+2}}=
\int_{-\pi/2}^{\pi/2}\frac{2dt}{\sqrt{2\cos 4t+b}},
          \\
q_4(a) &:= \int_0^\infty
          \frac{ds}{\sqrt{s^4-s^2+(b+2)/16}}.
\end{align*}
\end{lemma}

We define two $3\times 4$ matrices
$$
P_k:=
\Re \left(
\oint_{\gamma_k}
\!\!e^{i\theta}(\Phi_0)^T,
\oint_{\phi_3\circ\gamma_k}
\!\!e^{i\theta}(\Phi_0)^T,
\oint_{(\phi_3)^2\circ\gamma_k}
\!\!e^{i\theta}(\Phi_0)^T,
\oint_{(\phi_3)^3\circ\gamma_k}
\!\!e^{i\theta}(\Phi_0)^T 
\right)
$$
for $k=1,2$.  
Then 
$f_{a,\theta}$ is triply periodic if and
only if the eight column vectors of 
$(P_1,P_2)$ belong to some lattice of $\R^3_1$.

Now we consider the case where $\theta=0$.
Since 
$$
\oint_{(\phi_3)^j\circ\gamma_k}
\!\!(\Phi_0)^T
=
\oint_{\gamma_k}
\!\!((\phi_3)^j)^*(\Phi_0)^T \qquad (j=1,2,3;\,\,k=1,2),
$$
Lemma \ref{lm:phi1-4} yields that
\begin{align}
\label{eq:P1}
\left. P_1\right|_{\theta=0}&=
\left (\begin{array}{rrrr}
         -q_1 & q_1 & -q_1 &  q_1 \\
	  q_2 & q_2 & -q_2 & -q_2 \\
	  q_2 &-q_2 & -q_2 &  q_2
       \end{array}\right),\\
\label{eq:P2}
\left. P_2\right|_{\theta=0}
&=
\left (\begin{array}{rrrr}
           0   &  0  &  0  &  0 \\
	   0   &  q_2  & 0 & -q_2\\
	   q_2 &  0  & -q_2 & 0 
       \end{array}\right),
\end{align}
where $q_j=q_j(a)$ $(j=1,\dots,4)$ are as in Lemma~\ref{lm:period-triply}. 
Since 
each column vector of $\left. P_1\right|_{\theta=0}$
and $\left. P_2\right|_{\theta=0}$ is contained 
in the lattice
\begin{equation}
\label{eq:P3}
\Lambda
:=
\left\{m_0\begin{pmatrix}q_1\\
                       0\\
                       0\end{pmatrix}+
       m_1\begin{pmatrix}0\\
                      q_2\\
                      0\end{pmatrix}+
       m_2\begin{pmatrix} 0\\
                        0\\
                        q_2
          \end{pmatrix}\,;\, m_0,m_1,m_2\in\Z
      \right\}, 
\end{equation}
the surface
$$
f_{a,0}:M_a\longrightarrow \R^3_1/\Lambda 
$$
gives a maximal surface for all $a\in (0,1)$.  
The left hand side of Figure \ref{fig:surface_PD} 
is the figure of $f_{a,0}$ for  
 $a=(\sqrt{3}-1)/\sqrt{2}$.
 
Now we consider the case where $\theta=\pi/2$. By similar computations,
we have that
$$
\left. P_1\right|_{\theta=\pi/2}=O,\quad
\left. P_2\right|_{\theta=\pi/2}
=
 \left(\begin{array}{rrrr}
               -q_3 &  q_3  &  -q_3   & q_3 \\
	        q_4 &   0   &  -q_4   &  0  \\
	        0   & -q_4  &   0     & q_4
       \end{array}
 \right).
$$
Since each column of 
$\left. P_2\right|_{\theta=\pi/2}$ is contained in 
the lattice
$$
\Lambda'
:=
\left\{
      m_0\begin{pmatrix}q_3\\
                       q_4\\
               0\end{pmatrix}+
       m_1\begin{pmatrix}q_3\\
                      0\\
                      q_4\end{pmatrix}+
       m_2\begin{pmatrix}q_3\\
                      0\\
                      -q_4\end{pmatrix}\,;\,
		      m_0,m_1,m_2\in\Z
      \right\}, 
$$
the surface
$$
f_{a,\pi/2}:M_a\longrightarrow \R^3_1/\Lambda' 
$$
gives a maximal surface for all $a\in (0,1)$.  
The right hand side of Figure \ref{fig:surface_PD} 
corresponds to the figure of $f_{a,\pi/2}$
for $a=(\sqrt{3}-1)/\sqrt{2}$.

\begin{remark}\label{re:gyroid}
Numerical experiments suggest that there exists a   
triply periodic member in the family with $\theta \in (0, \pi/2)$, as an analogue of the Gyroid, which appears to have no self-intersections.
See Figure~\ref{fig:gyroid}. It would be interesting to theoretically confirm this observation.   
\begin{figure}
\begin{center}
\begin{tabular}{ccc}
 \includegraphics[width=.35\linewidth]{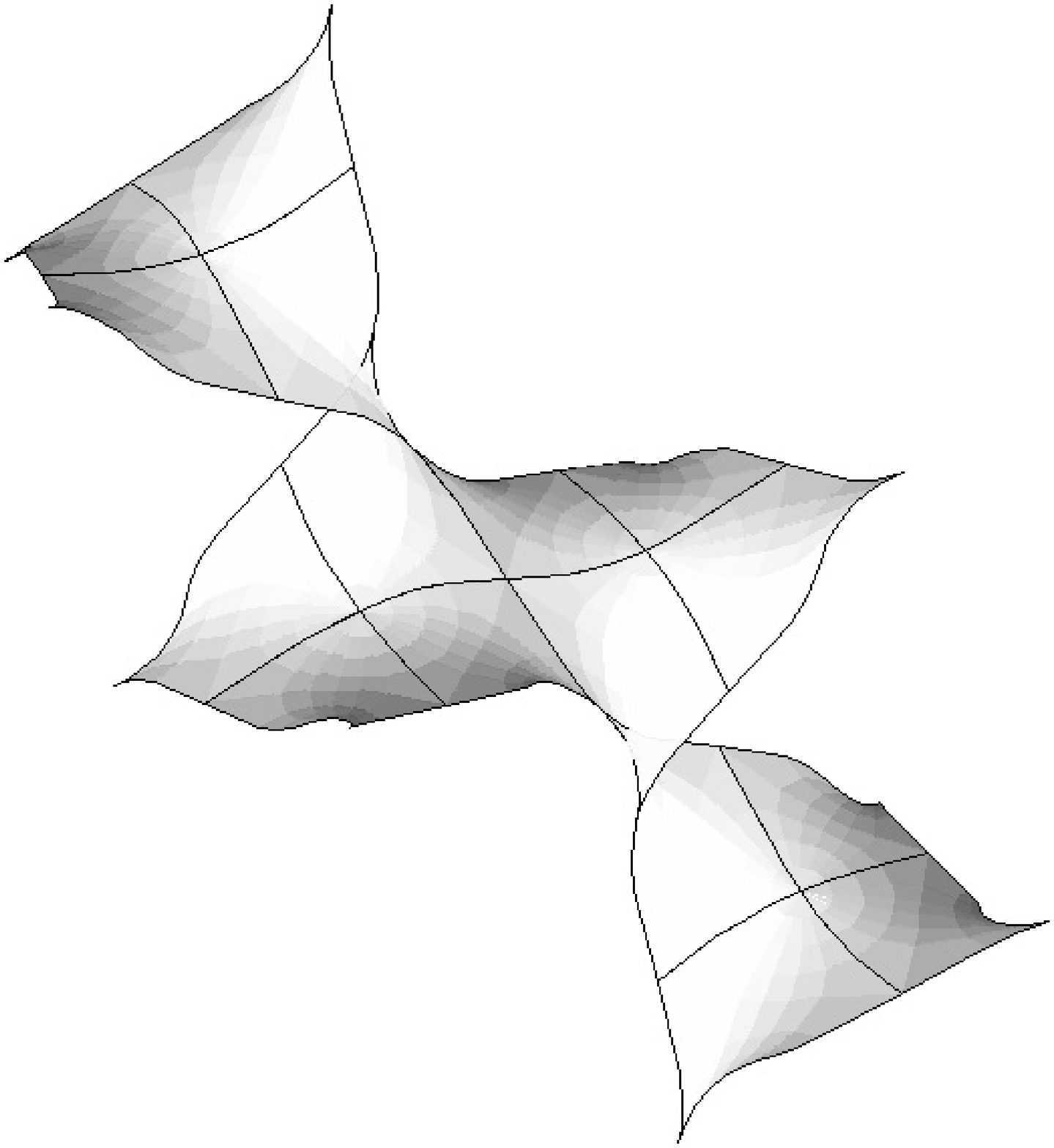}
& \hspace*{20pt} &
 \includegraphics[width=.45\linewidth]{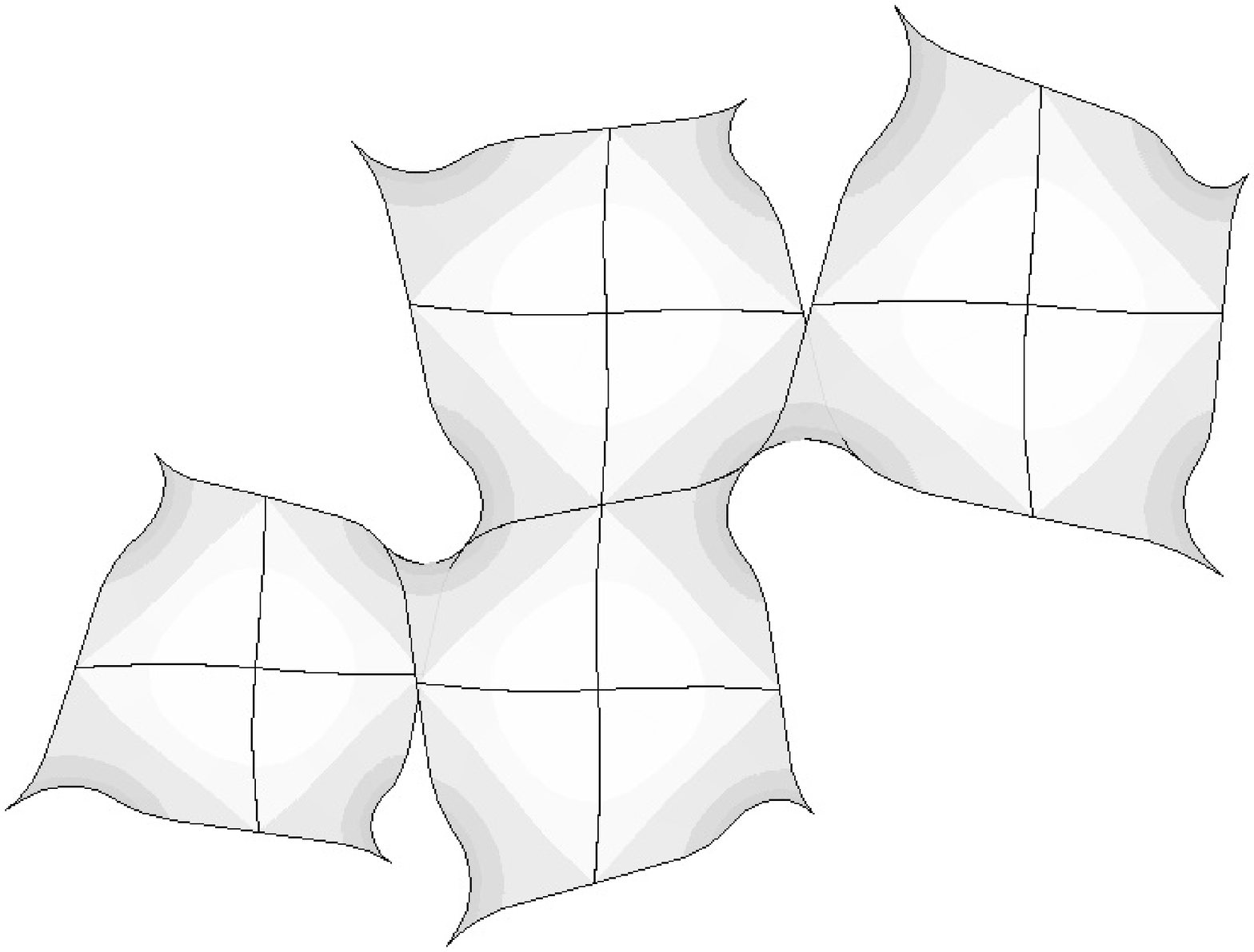} \\
  \unitlength=1pt
 \begin{picture}(100,0)
 \put(145,10){\vector(1,0){15}}
 \put(145,10){\vector(0,1){15}}
 \put(167,10){\makebox(0,0)[cc]{$x_1$}}
 \put(145,30){\makebox(0,0)[cc]{$x_2$}}
 \end{picture} & & 
\end{tabular}
\end{center}
\caption{
Two different views of the Gyroid-type maximal surface with $a\approx 0.346014$ 
and $\theta\approx 0.73073\approx 41.8685^\circ$ mentioned in Remark~\ref{re:gyroid}.
}
\label{fig:gyroid}
\end{figure} 
\end{remark}

\begin{remark}\label{rm:max-lim-1}
Here we consider the limit of $f_{a,\theta}$
as  $a\to 1$.  
The Riemann surface $M_a$ collapses to two spheres with 
four singular points at $(z,w)=(\pm e^{\pm\pi i/4},0)$, and 
the limit of $f_{a,\theta}$ 
is divided into two congruent maximal surfaces 
with the Weierstrass data 
\[
              G=z,\qquad \eta_\theta =\pm e^{i\theta}\frac{dz}{z^4+1}
              \qquad \bigl(\theta\in [0,\pi )\bigr)
\]
on $M':=(\C\cup\{\infty\})\setminus\{\pm e^{\pm\pi i/4}\}$.  
The limit of $f_{a,0}$ is a subset of the 
triply periodic
real analytic maximal surface
$$
  \mathcal S_+:=\{(t,x,y)\in \R^3_1\,;\, \cos t =\cos x \cos y\}
$$
called {\it spacelike  Scherk surface}, 
which contains singular lightlike lines
 (see \cite{CR} and \cite{CR2} for the whole figure of $\mathcal S_+$). 
On the other hand, the limit of $f_{a,\pi/2}$ is a subset of
the zero mean curvature entire graph
$$
  \mathcal S_0:=\{(t,x,y)\in \R^3_1\,;\, 
  e^t \cosh x=\cosh y 
\},
$$
given by Osamu Kobayashi \cite{K} (see also \cite{CR} and \cite{CR2}). 
$\mathcal S_0$ also contains four
disjoint timelike minimal surfaces as subsets. 
See Figure \ref{fig:scherk}.

\begin{figure}
\begin{center}
\begin{tabular}{cc}
 \includegraphics[width=.40\linewidth]{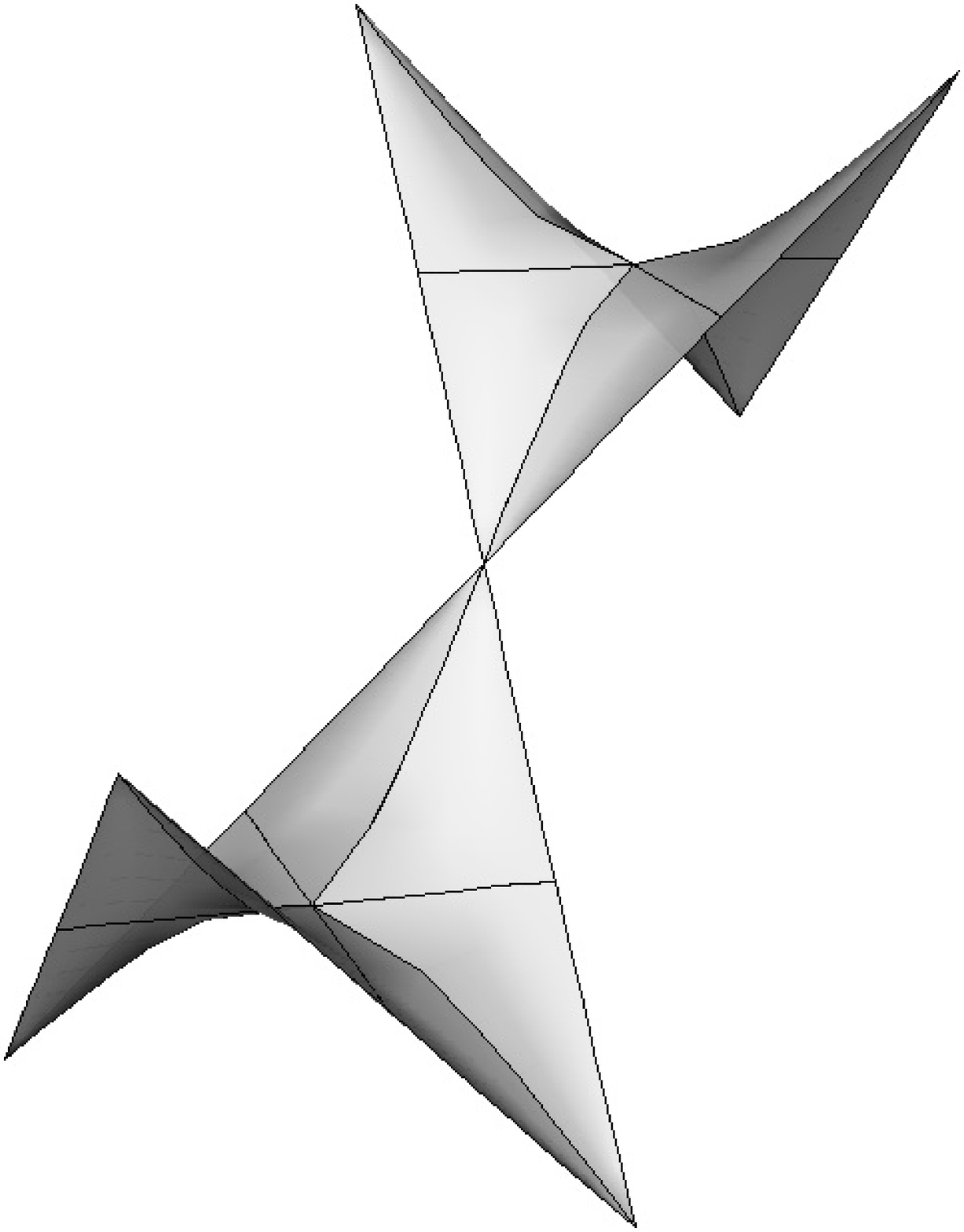}&
 \includegraphics[width=.40\linewidth]{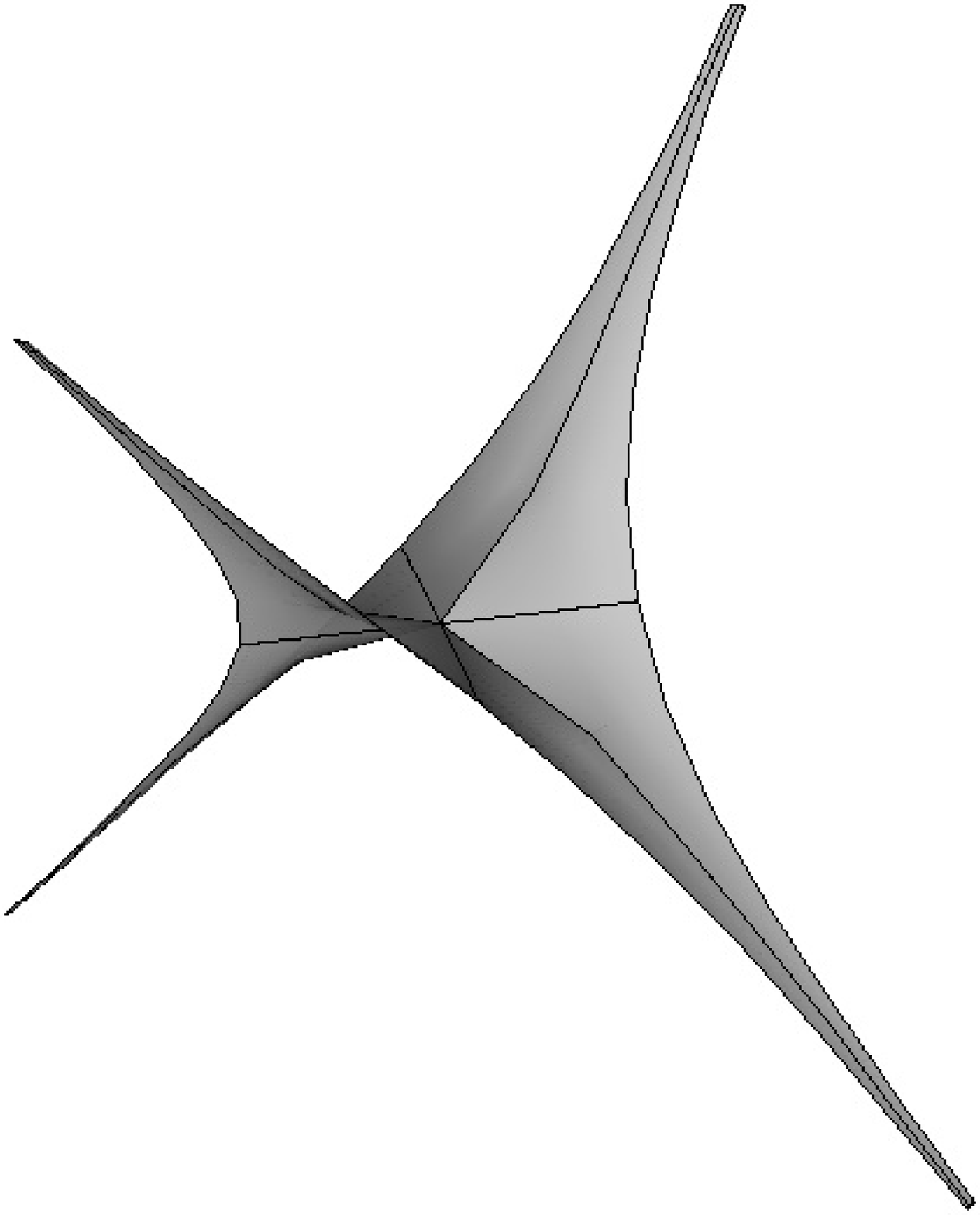} \\
\end{tabular}
\end{center}
\caption{
  The spacelike Scherk surface (left) and the spacelike part of the surface $\mathcal S_0$
  in Remark~\ref{rm:max-lim-1} (right).
}
\label{fig:scherk}
\end{figure} 
\end{remark}

\begin{remark}\label{rm:max-lim-0}
Here we consider the limit of $f_{a,\theta}$
as  $a\to 0$. We first rescale the surface as $\sqrt{a^4+a^{-4}}f_{a,\theta}$ 
and then take the limit as $a\to 0$.  
The Riemann surface $M_a$ collapses as  $a\to 0$ to two spheres with 
two singular points at $(z,w)=(0,0), (\infty,\infty)$, and 
the limit of $f_{a,\theta}$ 
is divided into two congruent maximal surfaces 
with the Weierstrass data 
\[
              G=z,\qquad \eta_\theta =\pm e^{i\theta}\frac{dz}{z^2}
              \qquad \bigl(\theta\in [0,\pi )\bigr)
\]
on $M':=\C\setminus\{0\}$.  
The limits of $f_{a,0}$ and of $f_{a,\pi/2}$ as $a\to 0$ are the {\it spacelike elliptic catenoid}
and the {\it spacelike elliptic helicoid, respectively.}  
See Figure \ref{fig:cat-hel}.

\begin{figure}
\begin{center}
\begin{tabular}{cc}
 \includegraphics[width=.40\linewidth]{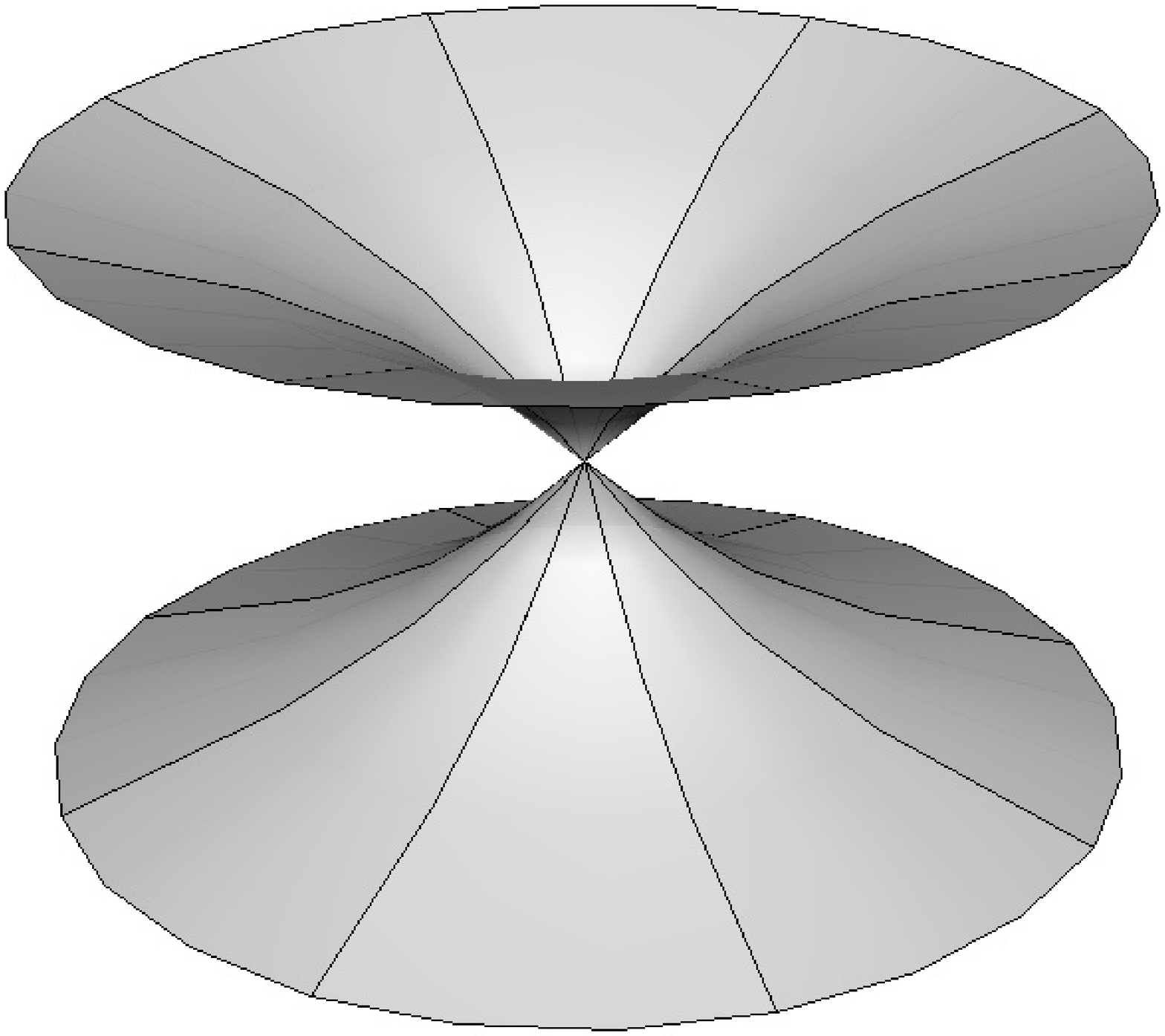}&
 \includegraphics[width=.40\linewidth]{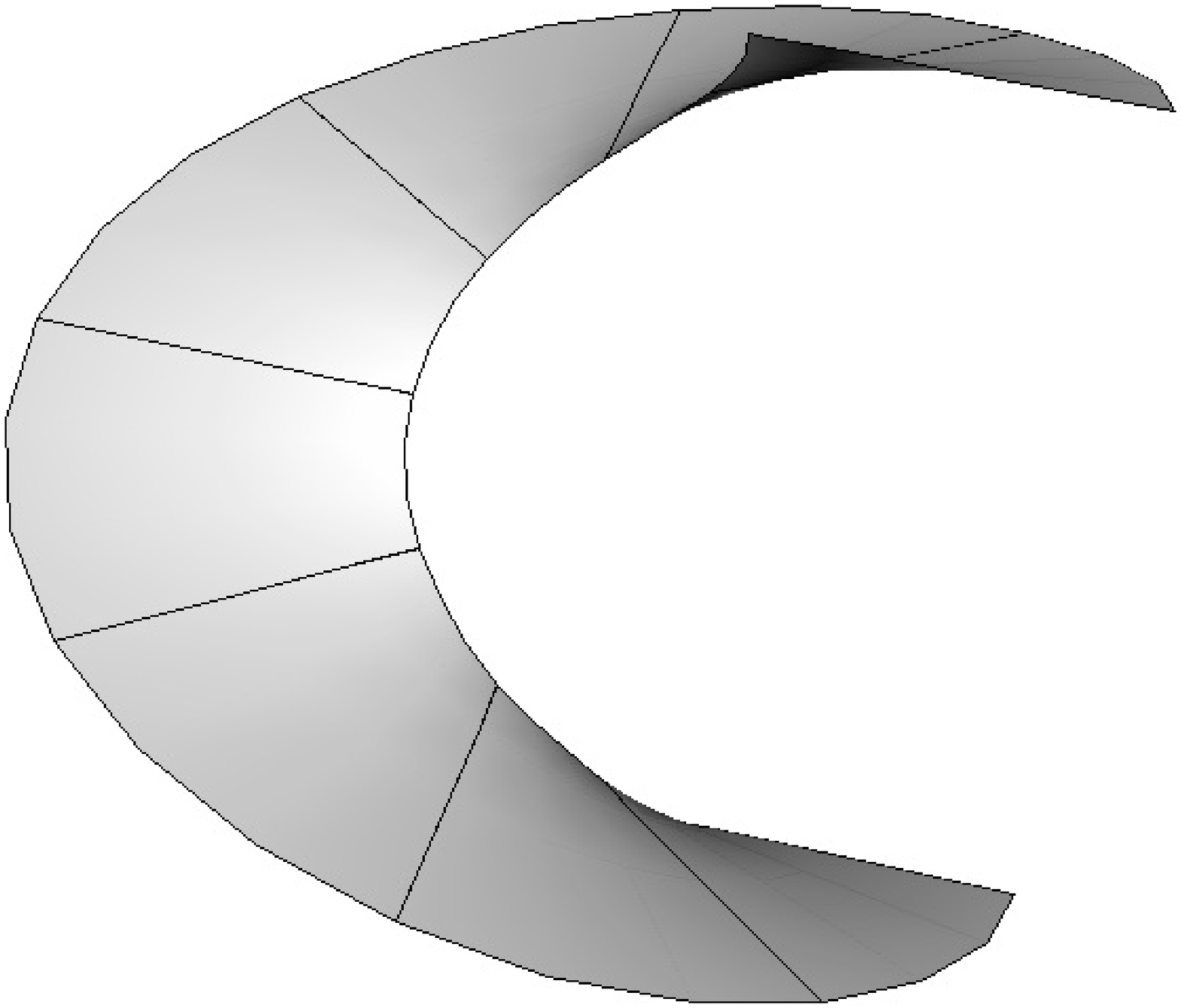} \\
\end{tabular}
\end{center}
\caption{
  The spacelike elliptic catenoid (left) and the spacelike elliptic helicoid (right).
}
\label{fig:cat-hel}
\end{figure} 
\end{remark}

\section{Analytic extensions of Schwarz D-type maximal surfaces 
to triply periodic zero mean curvature surfaces}
\label{sec:extension}

When a maximal surface has fold singularities, one can analytically  
extend the maximal surface to 
a timelike surface with mean curvature zero.  
This fact has been observed in 
\cite[Theorem 2.13]{CR2}. 
In the previous section,
we observed that 
the Schwarz D-type surface $f_{a,\pi/2}$
admits only fold singularities for each $0<a<1$.
The image of the singular set of 
$f_{a,\pi/2}$
is a lightlike curve
\begin{equation}\label{eq:gamma}
\gamma_a (s) := \int_0^s \xi_a (t) \left( 1,\, -\cos t\ ,\, -\sin t \right) \,dt,
\qquad \left(\xi_a(t):=\frac{2}{\sqrt{2\cos 4 t +a^4+a^{-4}}}\right).
\end{equation}
Then 
\[
  \tilde{f}_a (u,v):=\frac{1}{2}\left(\gamma_a (u+v) + \gamma_a (u-v) \right),
\]
is a timelike minimal surface 
(that is, a timelike surface with mean curvature zero,\,\, 
see Figure \ref{fig:mixface}) such that
\begin{equation}\label{eq:fold}
  \tilde{f}_a(u,0)=\gamma_a(u), 
\end{equation}
and $\tilde f_a$ is the analytic extension of the maximal surface $f_{a,\pi/2}$
(see Section 2 of \cite{CR2}). 

The following assertion holds.

\begin{lemma}\label{lm:f*imm}
$\tilde{f}_a(u,v)$ is an immersion on $\R\times (0,\pi)$. 
\end{lemma}

\begin{proof}
Since
\[
\frac{\partial \tilde{f}_a}{\partial u}=
    \frac{1}{2}\left(\gamma_a'(u+v)+\gamma_a'(u-v)\right),\quad
\frac{\partial \tilde{f}_a}{\partial v}
      =\frac{1}{2}\left(\gamma_a'(u+v)-\gamma_a'(u-v)\right), 
\]
$(u,v)$ is a singular point of $\tilde f_a$
(that is, a point where $\tilde f_a$ is not an immersion)
if and only if 
\[
\gamma_a'(u+v)=\xi_a(u+v) \left( 1,\, -\cos(u+v),\, -\sin (u+v) \right)
\]
and 
\[
\gamma_a'(u-v)=\xi_a(u-v) \left( 1,\, -\cos(u-v),\, -\sin (u-v) \right)
\]
are linearly dependent, where $\gamma_a'$
is the derivative of the curve $\gamma_a$.
The linear dependency 
of two vectors $\gamma_a'(u+v)$ and $\gamma_a'(u-v)$
is equivalent to the validity of the two equalities 
$$\cos(u+v)=\cos(u-v)  \quad\text{and}\quad		\sin(u+v)=\sin(u-v),$$
that is, $v\equiv 0 \pmod{\pi}$, proving the lemma.
\end{proof}

\begin{lemma}\label{lm:st-lines}
The timelike surface $\tilde f_a$ contains 
three line segments.
More precisely, 
\begin{enumerate}
\item\label{vertical} $\tilde f_a(u,\pi/2)$ $(u\in \R)$
is a straight line  parallel to the $x_0$-axis.
\item\label{horizontal1} $\tilde f_a(0,v)$ $(0<v<\pi)$
is a line segment parallel to the $x_2$-axis.
\item\label{horizontal2} $\tilde f_a(\pi/4,v)$ 
$(0<v<\pi)$ is a line segment  
parallel to the line $\{ x_0 =x_1 + x_2 = 0 \}$.
\end{enumerate}
\end{lemma}

\begin{proof}
By \eqref{eq:gamma}, 
\begin{align*}
\frac{\partial \tilde f_a}{\partial u}(u,\pi/2)
&= \frac{1}{2}\left(\gamma_a'(u+\pi/2)+\gamma_a'(u-\pi/2)\right) \\
&= \frac{1}{2}\xi_a(u+\pi/2)\left( 1,\, -\cos(u+\pi/2),\, -\sin (u+\pi/2) \right) \\
&\quad + \frac{1}{2}\xi_a(u-\pi/2)\left( 1,\, -\cos(u-\pi/2),\, -\sin (u-\pi/2) \right) \\
&= \xi_a(u)(1,\,0,\,0),
\end{align*}
because $\xi_a(u+\pi/2)=\xi_a(u-\pi/2)=\xi_a(u)$. 
Thus \ref{vertical} is proved. 
Similarly, direct computations show 
\begin{align*}
\frac{\partial \tilde f_a}{\partial v}(0,v) &= \dfrac{-2\sin v}{\sqrt{2\cos 4v+a^4+a^{-4}}}(0,\,0,\,1), \\
\frac{\partial \tilde f_a}{\partial v}(\pi/4,v) &= \dfrac{\sqrt{2}\sin v}{\sqrt{-2\cos 4v+a^4+a^{-4}}}(0,\,1,\,-1).  
\end{align*}
Thus \ref{horizontal1} and \ref{horizontal2} hold.  
\end{proof}

Like minimal surfaces in $\R^3$, both spacelike maximal surfaces and timelike minimal surfaces 
have reflection principles as follows. 

\begin{fact}[cf. {\cite[Theorem 3.10]{ACM} and \cite[Lemmas 4.1 and 4.2]{KKSY}}]\label{fc:reflection}\ \par
\begin{enumerate}
\item Suppose a spacelike maximal surface contains a spacelike line. 
      Then the surface is symmetric with respect to the line. 
\item Suppose a spacelike maximal surface is perpendicular to a timelike plane. 
      Then the surface is symmetric with respect to the plane. 
\item Suppose a timelike minimal surface contains a spacelike line or a timelike line. 
      Then the surface is locally symmetric with respect to the line. 
\item Suppose a timelike minimal surface is perpendicular to a spacelike plane or a timelike plane. 
      Then the surface is locally symmetric with respect to the plane. 
\end{enumerate}
\end{fact}

We know that $\tilde f_a(u,0)$ ($u\in\R$) 
consists of fold singularities (cf. \eqref{eq:fold}).
Since $\tilde f_a(u,\pi/2)$ ($u\in\R$) is a straight line, 
(3) of Fact~\ref{fc:reflection} implies that each 
point of $\tilde f_a(u,\pi)$ ($u\in\R$) is also a fold singularity, 
and we can analytically extend $\tilde f_a$ to the Schwarz D-type maximal surface, 
by Lemma~\ref{lm:st-lines}. 
Also,  by Lemma~\ref{lm:st-lines}, we can consider 
\begin{equation}\label{eq:timelike-fp}
\Omega^\text{min}_{a}:=\{\tilde f_a(u,v)\in\R^3_1\,;\,0\le u\le\pi/4,\,0<v\le\pi/2\}
\end{equation}
to be a fundamental piece of $\tilde f_a$, 
because the whole timelike minimal immersion $\tilde f_a(u,v)$ ($u\in\R$, $0<v<\pi$) 
can be obtained by reflections of $\Omega^\text{min}_{a}$.
Note, by Lemma~\ref{lm:f*imm}, that $\Omega^\text{min}_{a}$ is immersed.
The boundary $\partial\Omega^\text{min}_{a}$ of 
$\Omega^\text{min}_{a}$ consists of three straight line segments
\begin{align*}
\mathcal{L}_A^\text{min}&:=\{\tilde f_a(0,v)\in\R^3_1\,;\,0<v\le\pi/2\},\\
\mathcal{L}_B^\text{min}&:=\{\tilde f_a(\pi/4,v)\in\R^3_1\,;\,0<v\le\pi/2\},\\
\mathcal{L}_C^\text{min}&:=\{\tilde f_a(u,\pi/2)\in\R^3_1\,;\,0\le u\le\pi/4\},
\end{align*}
and the singular curve $\gamma_a(s)$ ($0\le s \le \pi/4$).

\medskip
\noindent
{\it Proof of Theorem A.}
For simplicity, 
we denote $f_{a,\pi/2}$ by $f_a$, where $f_{a,\pi/2}$ was defined in Section \ref{sec:triply}. 

By Lemma \ref{lm:phi1-4} and Fact \ref{fc:reflection}, we can consider 
\begin{equation}\label{eq:spacelike-fp}
\Omega^\text{max}_{a}:=\{f_a(z)\in\R^3_1\,;\,0\le |z|<1,\,0\le \arg z \le \pi/4\}
\end{equation}
to be a fundamental piece of $f_a$. 
We note that $\Omega^\text{max}_{a}$ is immersed. 
The boundary $\partial\Omega^\text{max}_{a}$ of $\Omega^\text{max}_{a}$ 
consists of two straight line segments
which correspond to 
$$
\{z\in\C\,;\,0\le |z|<1,\,\arg z=0\} 
\quad\mbox{and}\quad 
\{z\in\C\,;\,a\le |z|<1,\,\arg z=\pi/4\},
$$
a planar curve  which corresponds to 
$$
\{z\in\C\,;\,0\le |z|\le a,\,\arg z=\pi/4\},
$$
and the singular curve $\gamma_a(s)$ ($0\le s \le \pi/4$).
We set
\begin{equation}\label{eq:0mega-ao}
\Omega_{a}^1:=\Omega^\text{max}_{a} \cup\{\gamma_a(s)\,;\,0\le s\le \pi/4\}\cup \Omega^\text{min}_{a}. 
\end{equation}
Since $\Omega^\text{max}_{a}$ and $\Omega^\text{min}_{a}$ match analytically 
through $\gamma_a(s)$ ($0\le s\le \pi/4$), 
$\Omega_{a}^1$ is immersed (see \cite[Section 2]{CR2} for the details). 
We define
\begin{align*}
\mathcal{L}_A^\text{max}&:=\{f_a(z)\in\R^3_1\,;\,0\le |z|<1,\,\arg z=0\}, \\
\mathcal{L}_B^\text{max}&:=\{f_a(z)\in\R^3_1\,;\,a\le |z|<1,\,\arg z=\pi/4\}, \\
\mathcal{L}_C^\text{max}&:=\{f_a(z)\in\R^3_1\,;\,0\le |z|\le a,\,\arg z=\pi/4\}.
\end{align*}
It can be easily checked that 
$\mathcal{L}_A^\text{max}$ is parallel to the $x_2$-axis and 
$\mathcal{L}_B^\text{max}$ is parallel to the line 
$$
\{(x_0,x_1,x_2)\in\R^3_1\,;\,x_0=0,\,x_1+x_2=0\}, 
$$
and 
$\mathcal{L}_C^\text{max}$ is contained in a plane which is parallel to the plane 
$$
\{(x_0,x_1,x_2)\in\R^3_1\,;\,x_1=x_2\}.
$$

Thus $\mathcal{L}_A^\text{max}$ and $\mathcal{L}_A^\text{min}$, as well as 
$\mathcal{L}_B^\text{max}$ and $\mathcal{L}_B^\text{min}$, are collinear. 

We set $\mathcal{L}_A:=\mathcal{L}_A^\text{max}\cup \mathcal{L}_A^\text{min}$ and 
$\mathcal{L}_B:=\mathcal{L}_B^\text{max}\cup \mathcal{L}_B^\text{min}$. 
Then the image of the projection of the boundary $\partial\Omega_{a}^1$ of $\Omega_{a}^1$ 
into the $x_1x_2$-plane is an isosceles right triangle. 
See Figure \ref{fig:fo1}. 
We denote this isosceles right triangle with its interior by $\Delta$. 
We also denote the length of the segment $\mathcal{L}_C^\text{min}$ by $|\mathcal{L}_C^\text{min}|$. 

We have already seen that $\Omega_a^1$ is immersed. 
Furthermore, we have the following proposition
which will be proved in Section~\ref{sc:proof-prop}. 

\begin{proposition}\label{pr:key-prop}
For each $a\in (0,1)$, $\Omega_a^1$ is embedded 
and contained in the closure of a vertical prism over the 
isosceles right triangle $\Delta$ with height $|\mathcal{L}_C^\text{min}|$. 
\end{proposition}

\begin{figure}
\begin{center}
\begin{tabular}{ccc}
 \includegraphics[width=.40\linewidth]{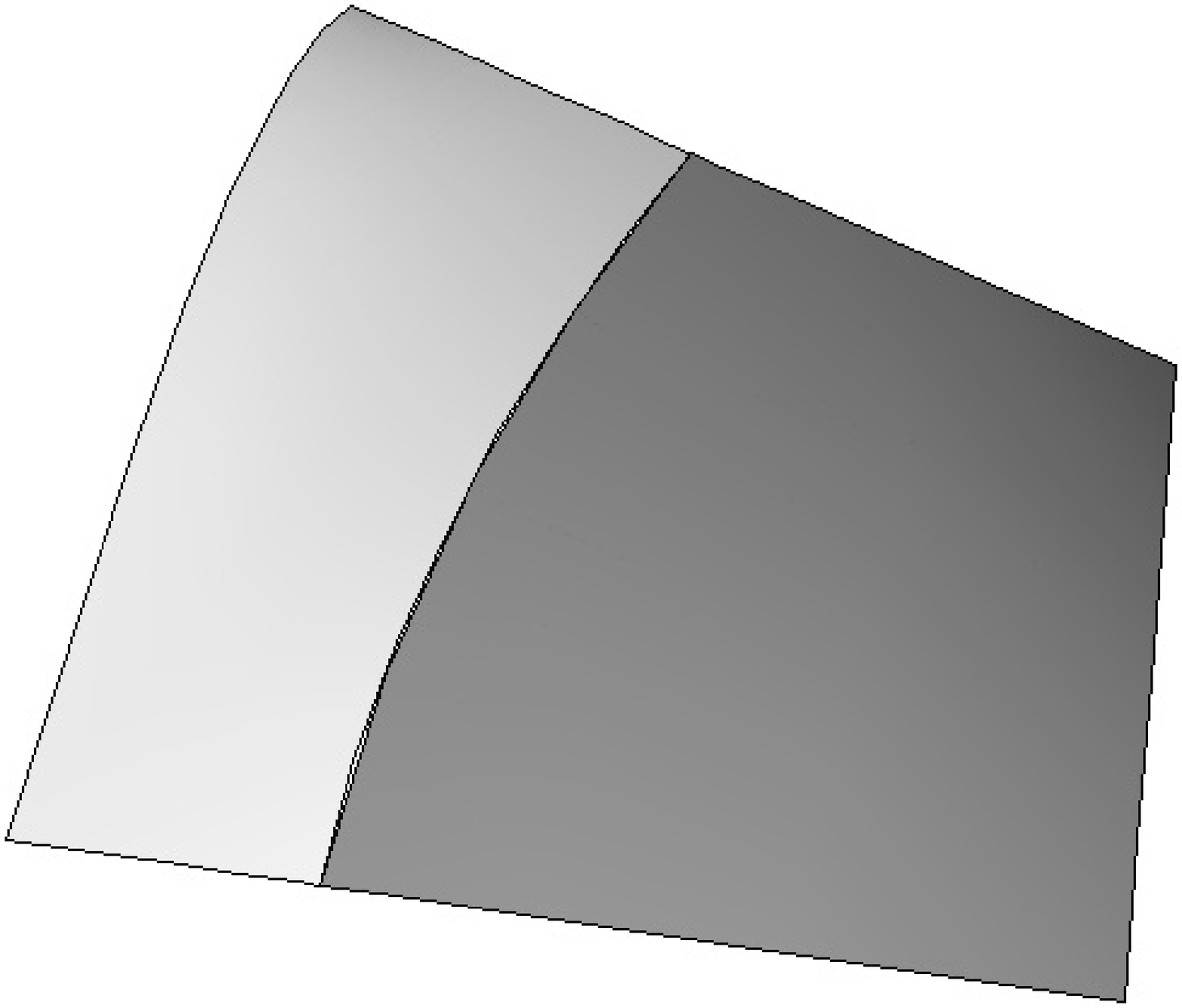} & \hspace*{20pt} &
 \includegraphics[width=.20\linewidth]{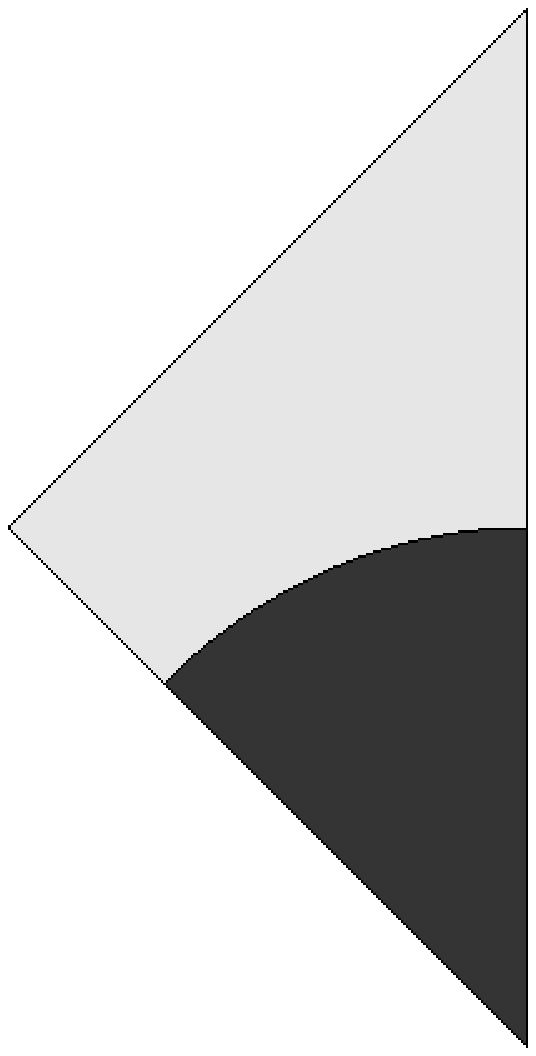} \\
 \unitlength=1pt
 \begin{picture}(100,0)
 \put(80,120){\makebox(0,0)[cc]{$\mathcal{L}_A$}}
 \put(30,20){\makebox(0,0)[cc]{$\mathcal{L}_B$}}
 \put(-20,70){\makebox(0,0)[cc]{$\mathcal{L}_C^\text{max}$}}
 \put(130,60){\makebox(0,0)[cc]{$\mathcal{L}_C^\text{min}$}}
 \put(240,100){\makebox(0,0)[cc]{$\mathcal{L}_A$}}
 \put(180,50){\makebox(0,0)[cc]{$\mathcal{L}_B$}}
 \put(190,130){\makebox(0,0)[cc]{$\mathcal{L}_C^\text{max}$}}
 \put(220,10){\makebox(0,0)[cc]{$\mathcal{L}_C^\text{min}$}}
 \put(170,0){\vector(1,0){15}}
 \put(170,0){\vector(0,1){15}}
 \put(192,0){\makebox(0,0)[cc]{$x_1$}}
 \put(170,20){\makebox(0,0)[cc]{$x_2$}}
 \end{picture} & & 
\end{tabular}
\end{center}
\caption{Left: $\Omega_{a}^1$ defined in \eqref{eq:0mega-ao}.  
The curve in the middle indicates the singular curve $\gamma_a(s)$, 
and the left hand side (resp. right hand side) is $\Omega^\text{max}_{a}$ 
(resp. $\Omega^\text{min}_{a}$).  
Right: Another view of $\Omega_{a}^1$ such that the 
line $\mathcal{L}_C^\text{min}$ is viewed as a single point at the bottom. 
On the top (resp. bottom) is $\Omega^\text{max}_{a}$ (resp. $\Omega^\text{min}_{a}$). 
}
\label{fig:fo1}
\end{figure} 

Now we extend $\Omega_a^1$ by reflection with respect to the planar curve $\mathcal{L}_C^\text{max}$. 
We denote the resulting surface by $\Omega_a^2$, which is two copies of $\Omega_a^1$. 
Then $\Omega_a^2$ is also embedded and the boundary consists of 
five straight line segments
($\mathcal{L}_B$ and its reflection are collinear).  
See Figure \ref{fig:fo2}. 

\begin{figure}
\begin{center}
\begin{tabular}{ccc}
 \includegraphics[width=.35\linewidth]{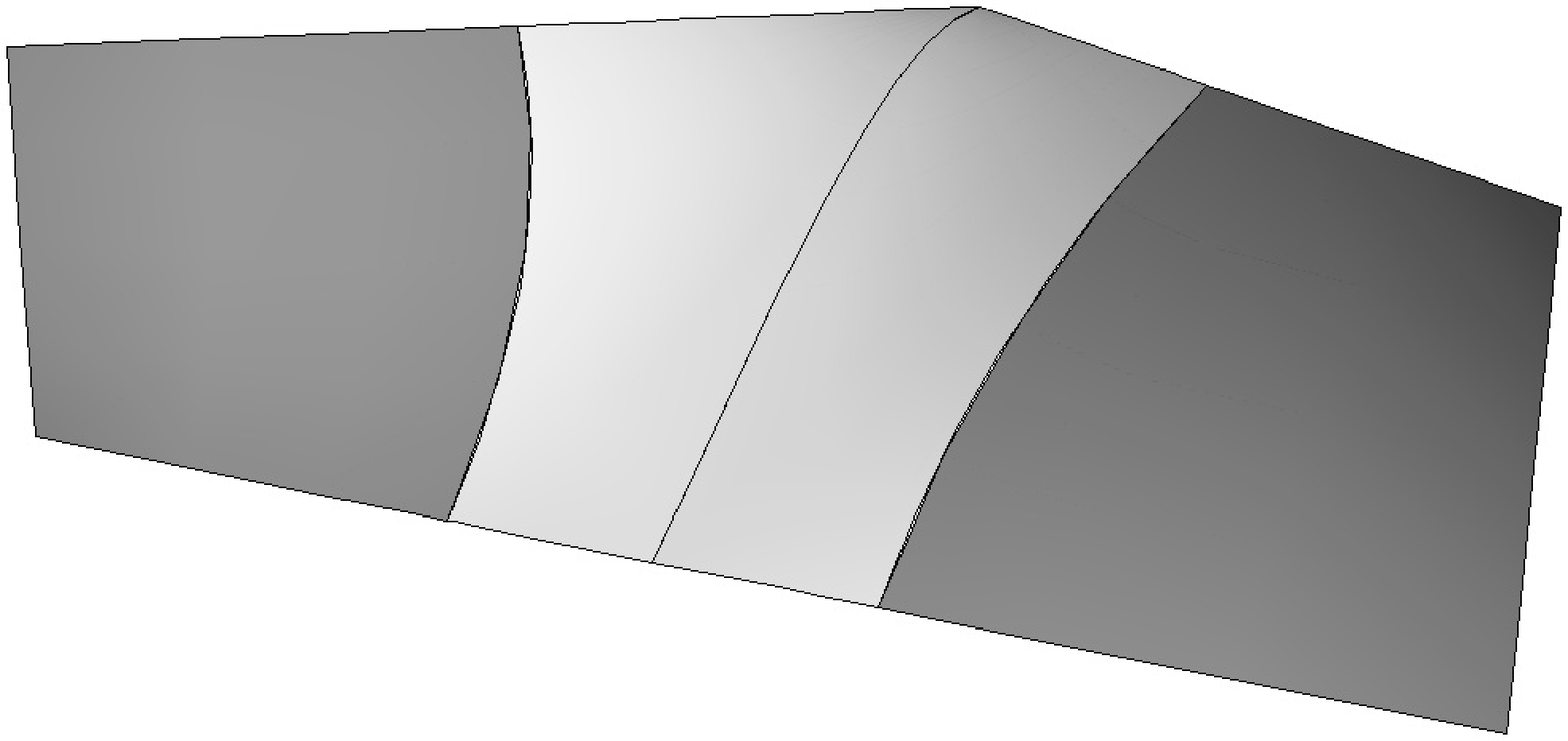} & \hspace*{20pt} &
 \includegraphics[width=.25\linewidth]{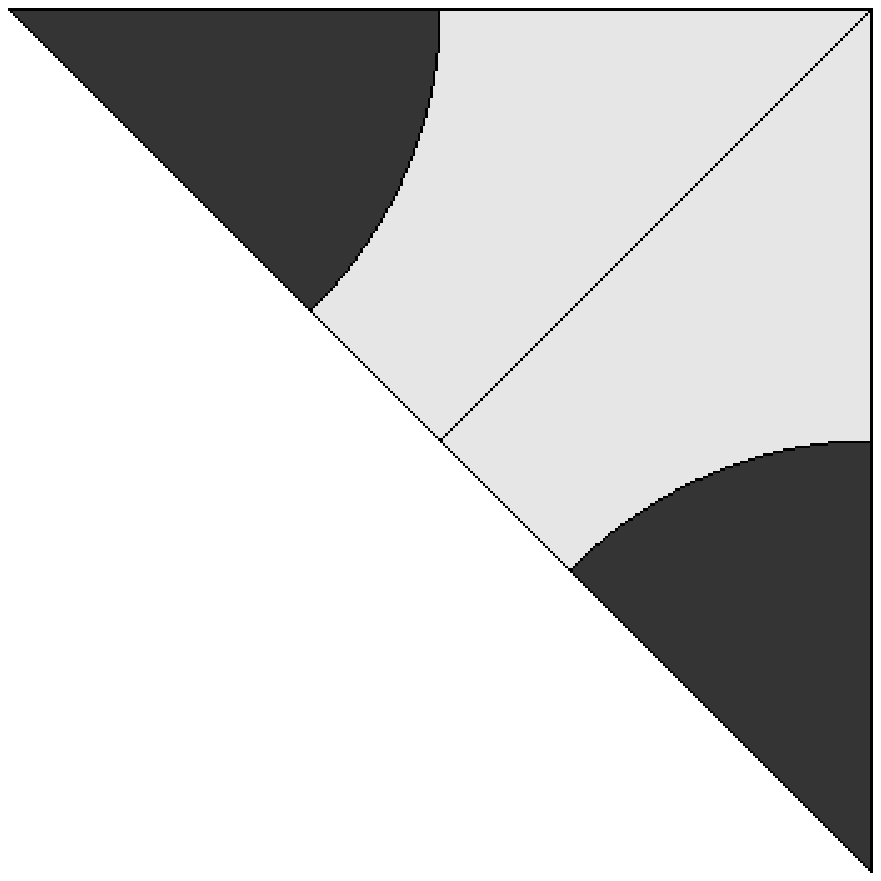} \\
 \unitlength=1pt
 \begin{picture}(100,0)
 \put(90,70){\makebox(0,0)[cc]{$\mathcal{L}_A$}}
 \put(30,80){\makebox(0,0)[cc]{$\mathcal{L}_A'$}}
 \put(50,15){\makebox(0,0)[cc]{$\mathcal{L}_B$}}
 \put(40,50){\makebox(0,0)[cc]{$\mathcal{L}_C^\text{max}$}}
 \put(125,40){\makebox(0,0)[cc]{$\mathcal{L}_C^\text{min}$}}
 \put(250,65){\makebox(0,0)[cc]{$\mathcal{L}_A$}}
 \put(205,110){\makebox(0,0)[cc]{$\mathcal{L}_A'$}}
 \put(190,50){\makebox(0,0)[cc]{$\mathcal{L}_B$}}
 \put(210,80){\makebox(0,0)[cc]{$\mathcal{L}_C^\text{max}$}}
 \put(240,5){\makebox(0,0)[cc]{$\mathcal{L}_C^\text{min}$}}
 \put(170,0){\vector(1,0){15}}
 \put(170,0){\vector(0,1){15}}
 \put(192,0){\makebox(0,0)[cc]{$x_1$}}
 \put(170,20){\makebox(0,0)[cc]{$x_2$}}
 \end{picture} & & 
\end{tabular}
\end{center}
\caption{Left: $\Omega_a^2$, that is, $\Omega_a^1$ and its reflection with respect to 
the plane of $\mathcal{L}_C^\text{max}$.  
The right hand side of $\mathcal{L}_C^\text{max}$ is $\Omega_a^1$ and the left hand side is its reflection. 
The spacelike parts are indicated by grey shades and the timelike parts by black shades. 
Right: Another view of $\Omega_a^2$. 
The right bottom (resp. left top) is $\Omega_a^1$ (resp. its reflection). }
\label{fig:fo2}
\end{figure} 

We denote the reflection of $\mathcal{L}_A$ by $\mathcal{L}_A'$.  

We extend $\Omega_a^2$ by two more reflections with respect to $\mathcal{L}_A$ and $\mathcal{L}_A'$. 
We denote the resulting surface by $\Omega_a^8$, which is four copies of $\Omega_a^2$. 
Then $\Omega_a^8$ is embedded and the boundary consists of eight straight line
segments  
(four (horizontal) spacelike line segments 
and four (vertical) timelike line segments).  
See Figure \ref{fig:fo8}. 

\begin{figure}
\begin{center}
\begin{tabular}{ccc}
 \includegraphics[width=.35\linewidth]{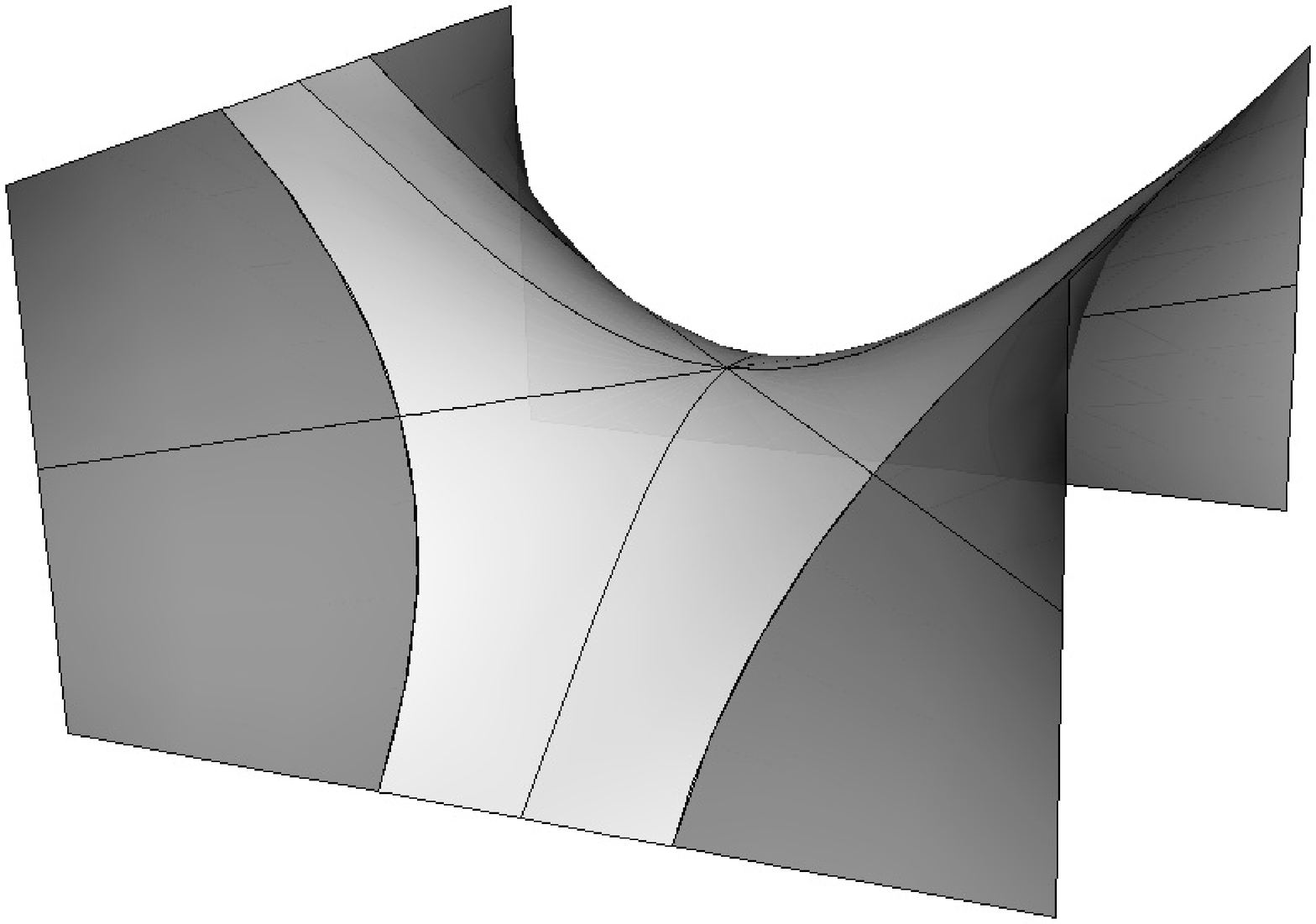} & \hspace*{20pt} &
 \includegraphics[width=.25\linewidth]{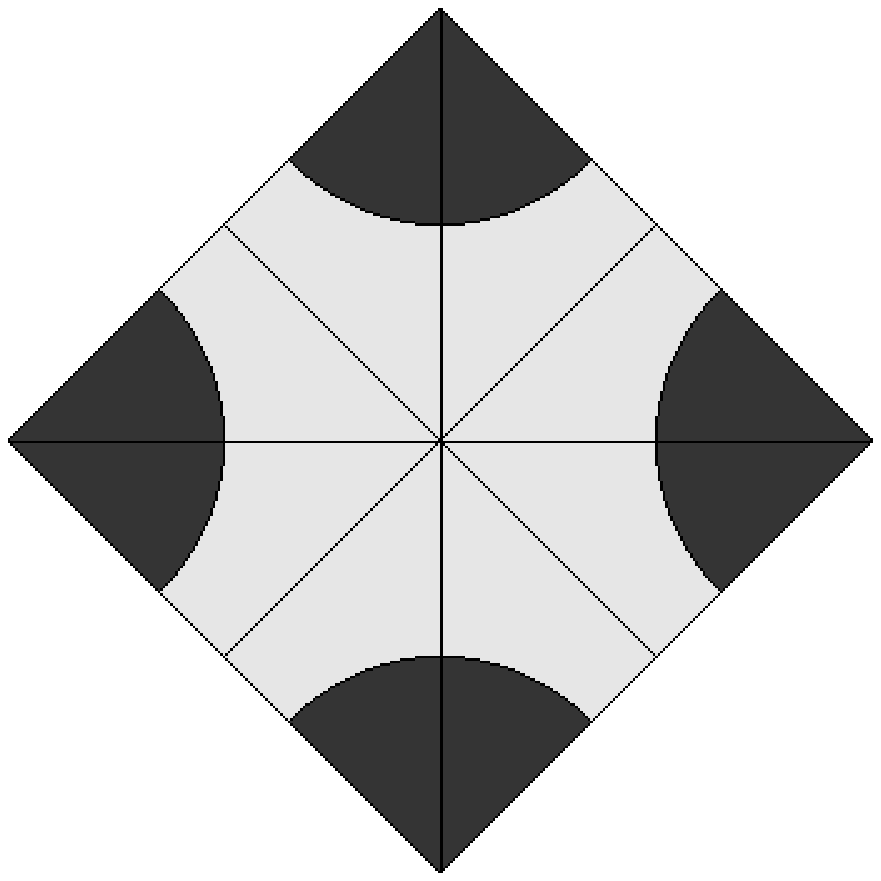} \\
  \unitlength=1pt
 \begin{picture}(100,0)
 \put(160,10){\vector(1,0){15}}
 \put(160,10){\vector(0,1){15}}
 \put(182,10){\makebox(0,0)[cc]{$x_1$}}
 \put(160,30){\makebox(0,0)[cc]{$x_2$}}
 \end{picture} & & 
\end{tabular}
\end{center}
\caption{Left: $\Omega_a^8$, that is, 
$\Omega_a^2$ with its reflections with respect to $\mathcal{L}_A$ and $\mathcal{L}_A'$.  
$\Omega_a^2$ is in the front.  
Right: Another view of $\Omega_a^8$.}
\label{fig:fo8}
\end{figure} 

We now rotate $\Omega_a^8$ with respect to the $x_0$ axis by angle $\pi/4$ 
so that the horizontal lines in the bottom 
(which are indicated by $\mathcal{L}_B$ in Figures~\ref{fig:fo1} and \ref{fig:fo2}) are parallel to the $x_1$ axis. 
Then the boundary $\partial\Omega_a^8$ of $\Omega_a^8$ consists of two (horizontal) 
line segments 
parallel to the $x_1$ axis in the bottom, 
two (horizontal) line segments parallel to the $x_2$ axis in the top, and 
four (vertical) line segments parallel to the $x_0$ axis. 
We label one of the (horizontal) line segments, parallel to the $x_1$ axis in the bottom, 
as $\hat{\mathcal{L}}_B$, 
and one of the (vertical) line segments, which is parallel to the $x_0$ axis and connects to 
$\hat{\mathcal{L}}_B$, as $\hat{\mathcal{L}}_C$. 
See Figure~\ref{fig:fo8'}. 

\begin{figure}
\begin{center}
\begin{tabular}{c}
 \includegraphics[width=.40\linewidth]{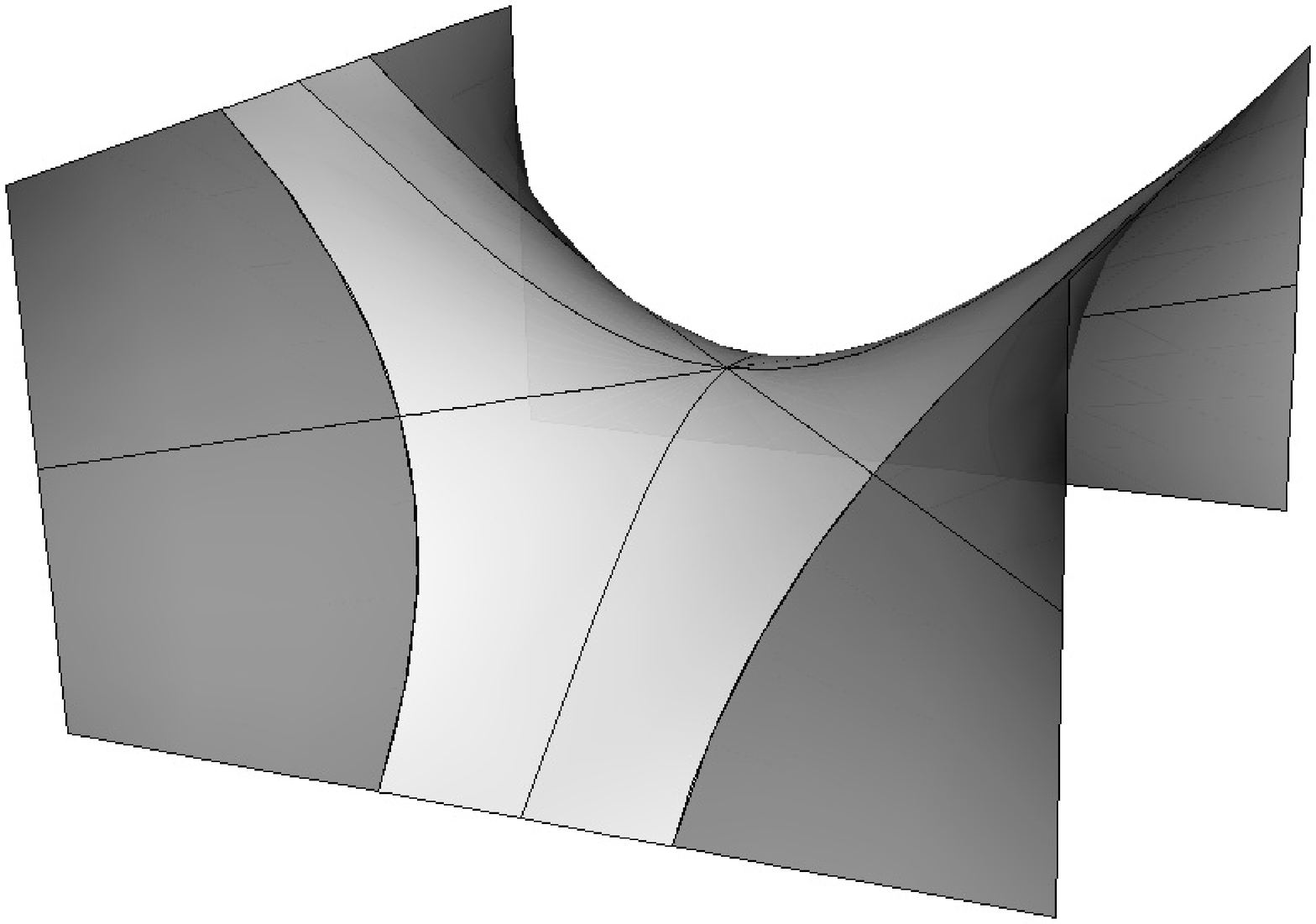} \\
 \unitlength=1pt
 \begin{picture}(100,0)
 \put(30,15){\makebox(0,0)[cc]{$\hat{\mathcal{L}}_B$}}
 \put(-30,60){\makebox(0,0)[cc]{$\hat{\mathcal{L}}_C$}}
 \end{picture} 
\end{tabular}
\end{center}
\caption{$\Omega_a^8$ with labels $\hat{\mathcal{L}}_B$ and $\hat{\mathcal{L}}_C$. }
\label{fig:fo8'}
\end{figure} 

We denote the length of the segment $\hat{\mathcal{L}}_B$ 
(resp. $\hat{\mathcal{L}}_C$) by $|\hat{\mathcal{L}}_B|$ (resp. $|\hat{\mathcal{L}}_C|$). 
We extend $\Omega_a^8$ by two more reflections with respect to $\hat{\mathcal{L}}_B$ and 
$\hat{\mathcal{L}}_C$. 
We denote the resulting surface by $\Omega_a^{32}$, which is four copies of $\Omega_a^8$. 
Then $\Omega_a^{32}$ is still embedded and is contained in the closure of 
a rectangular parallelepiped with height $2|\hat{\mathcal{L}}_C|$ over a square of side length $2|\hat{\mathcal{L}}_B|$. 
See Figure \ref{fig:T}. 

\begin{figure}
\begin{center}
\begin{tabular}{c}
 \includegraphics[width=.50\linewidth]{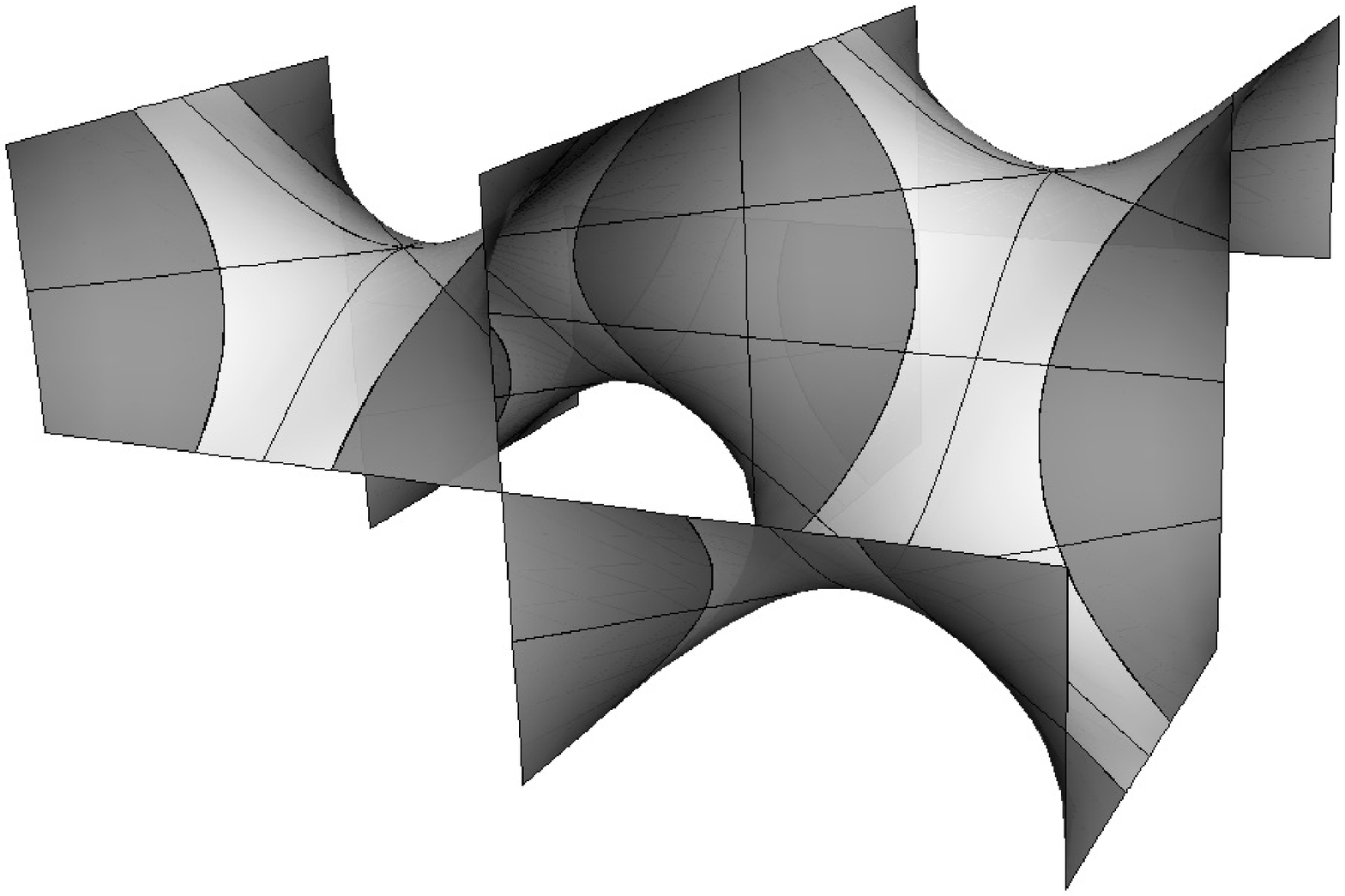} \\
 \unitlength=1pt
 \begin{picture}(100,0)
 \put(48,133){\vector(1,-1){10}}
 \put(85,80){\makebox(0,0)[cc]{$\hat{\mathcal{L}}_B$}}
 \put(41,137){\makebox(0,0)[cc]{$\hat{\mathcal{L}}_C$}}
 \end{picture} 
\end{tabular}
\end{center}
\caption{$\Omega_a^{32}$, that is, 
$\Omega_a^8$ with its reflections with respect to $\hat{\mathcal{L}}_B$ and $\hat{\mathcal{L}}_C$.}
\label{fig:T}
\end{figure} 

Then $\Omega_a^{32}$ and its translation by 
$$
\left(2\epsilon_0 |\hat{\mathcal{L}}_C|,\,
      2\epsilon_1 |\hat{\mathcal{L}}_B|,\,
      2\epsilon_2 |\hat{\mathcal{L}}_B|\right)
      \qquad\text{where}\quad \epsilon_j=\pm 1 \quad (j=0,1,2)
$$
match analytically, since each translation can be obtained by a reflection with respect to 
some straight line. 
Therefore, 
$$
\Omega_a:=
\left\{\Omega_a^{32}+(2m_0|\hat{\mathcal{L}}_C|,\,2m_1|\hat{\mathcal{L}}_B|,\,2m_2|\hat{\mathcal{L}}_B|)
\,\,;\,\, m_0,m_1,m_2\in\Z\right\}
\subset \R^3_1
$$
is an embedded triply periodic surface. 
In other words, $\Omega_a^{32}$ is embedded in a torus $\R^3_1/\Gamma_a$, 
where
$$
\Gamma_a:=\left\{(2m_0|\hat{\mathcal{L}}_C|,\,2m_1|\hat{\mathcal{L}}_B|,\,2m_2|\hat{\mathcal{L}}_B|)
\in\R^3_1\,\,;\,\, m_0,m_1,m_2\in\Z\right\}
$$
is a lattice in $\R^3_1$. 

We clearly see that this surface $\Omega_a$ is topologically the same as 
the Schwarz D minimal surface in $\R^3$  
(see Figure~\ref{fig:mixface}).
Thus $\Omega_a^{32}$ in the quotient $\R^3_1/\Gamma_a$ is a closed orientable 2-manifold of 
genus three. 
\qed

\begin{remark}\label{rm:mix-lim-1}
Here we consider the limit as $a\to 1$.  
In this case, we obtain the zero mean curvature entire graph
$$
  \mathcal S_0=\{(t,x,y)\in \R^3_1\,;\, 
  e^t \cosh x=\cosh y 
\},
$$
which we already mentioned in Remark~\ref{rm:max-lim-1}. 
See Figure \ref{fig:mix-scherk}.
See also Figure \ref{fig:min-scherk} to compare this limiting behavior with that of the minimal surfaces 
in $\R^3$. 

\begin{figure}
\begin{center}
\begin{tabular}{cc}
 \includegraphics[width=.49\linewidth]{surf-t-yb09.eps} & 
 \includegraphics[width=.35\linewidth]{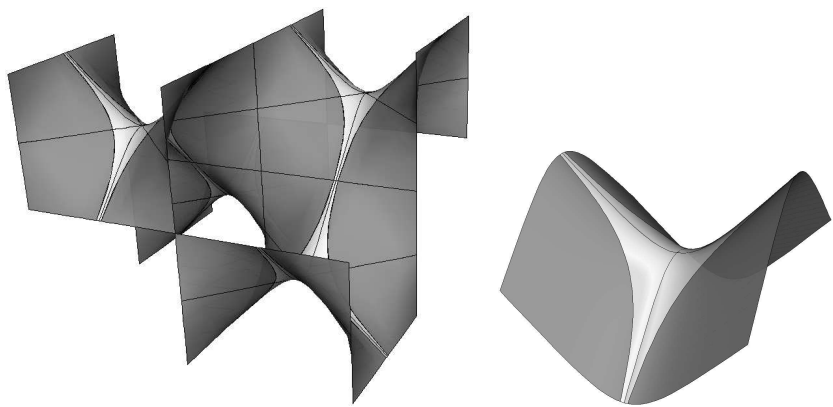}
 \end{tabular}
\end{center}
\caption{$\Omega_a^{32}$ with $a=0.9$ (left) and its limit as $a\to 1$
(right).}
\label{fig:mix-scherk}
\end{figure} 

\begin{figure}
\begin{center}
\begin{tabular}{cc}
 \includegraphics[width=.28\linewidth]{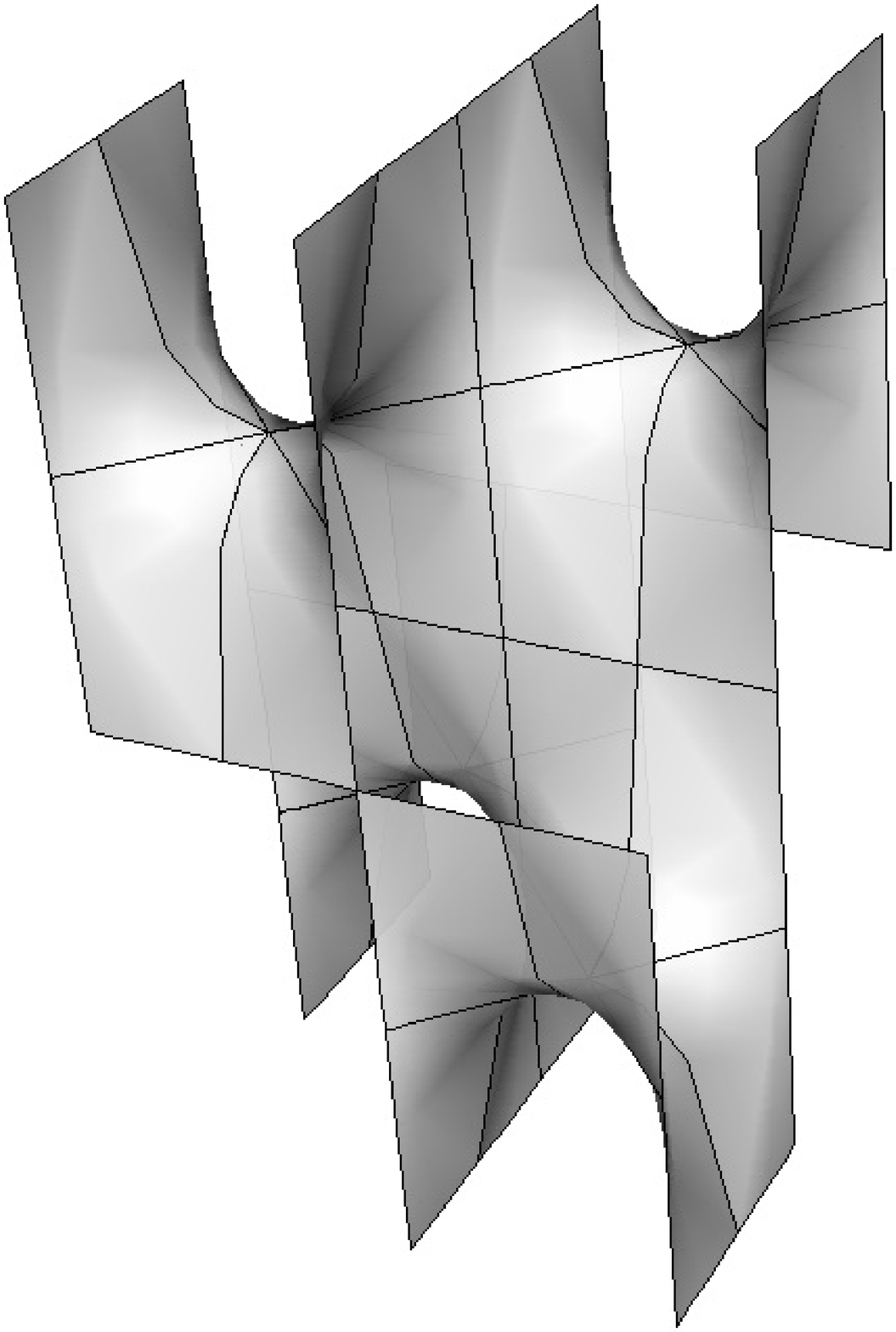} & 
 \includegraphics[width=.45\linewidth]{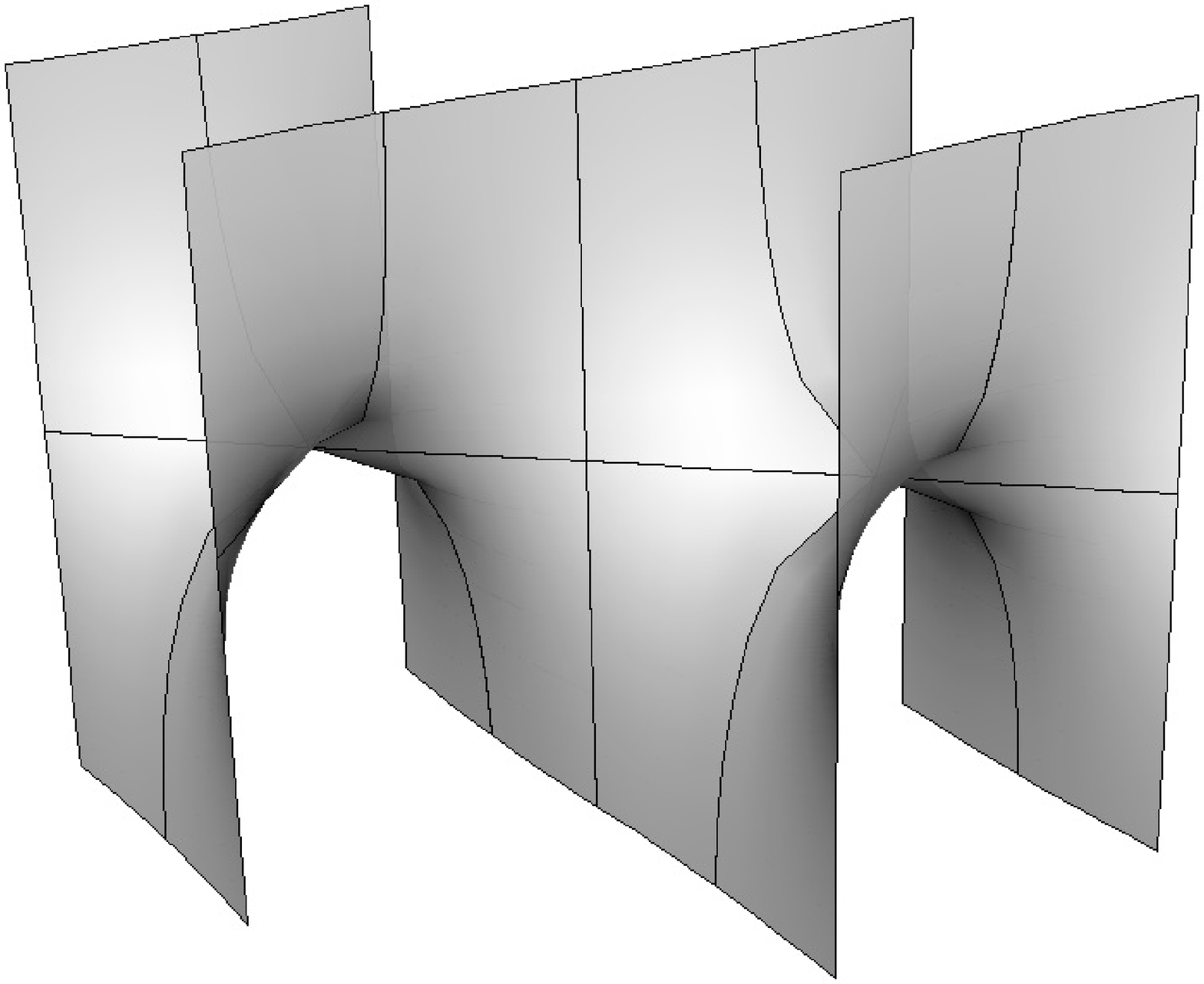}
 \end{tabular}
\end{center}
\caption{Schwarz D surface in $\R^3$ with $a=0.9$ (left) and the doubly periodic Scherk surface 
in $\R^3$ as a limit of Schwarz D surface as $a\to 1$ (right).}
\label{fig:min-scherk}
\end{figure} 
\end{remark}

\begin{remark}\label{rm:mix-lim-0}
Here we consider the  limit as $a\to 0$. We first multiply the surface by $\sqrt{a^4+a^{-4}}$ to rescale the surface, 
as we did in Remark~\ref{rm:max-lim-0},  
and then take the limit as $a\to 0$.  
In this case, we obtain the zero mean curvature surface which is exactly the same as the
minimal helicoid in $\R^3$ \cite{K}. 
See Figure \ref{fig:mix-hel}.
See also Figure \ref{fig:min-hel} to compare the limiting behavior with that of the minimal surfaces in $\R^3$. 

\begin{figure}
\begin{center}
\begin{tabular}{ccc}
 \includegraphics[width=.35\linewidth]{surf-t-yb01.eps} & 
 \includegraphics[width=.35\linewidth]{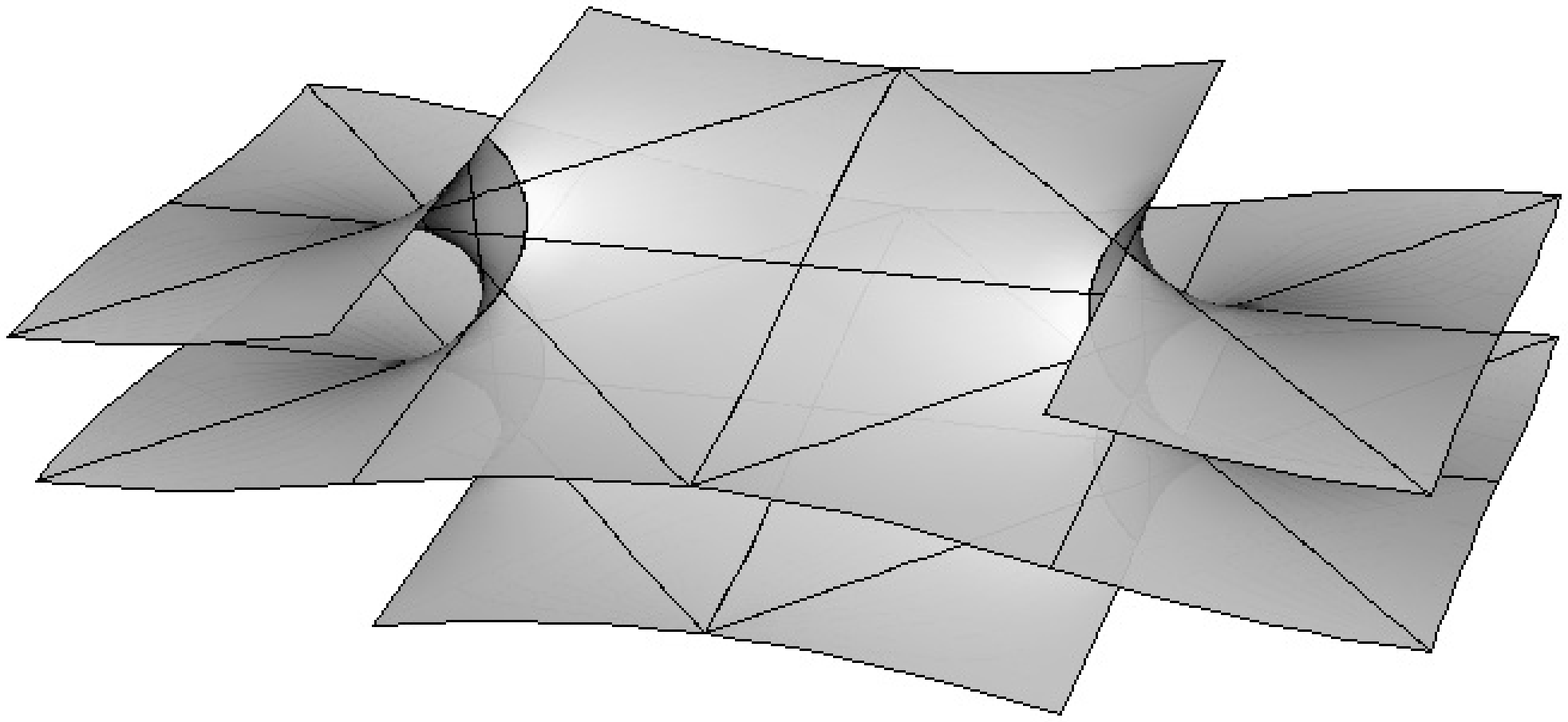} & 
 \includegraphics[width=.25\linewidth]{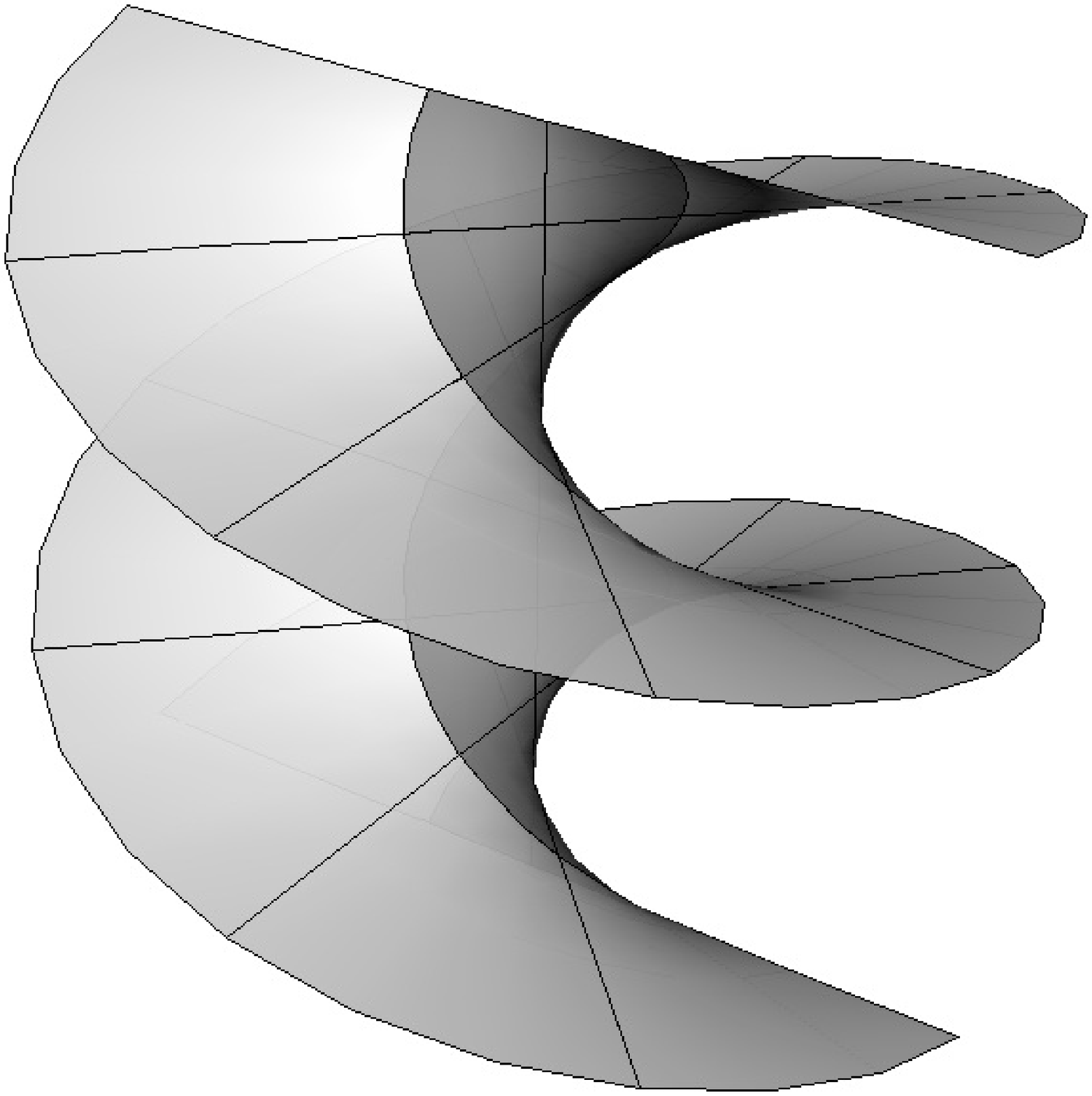} 
 \end{tabular}
\end{center}
\caption{$\Omega_a^{32}$ in $\R_1^3$ with $a=0.1$ (left), 
another view of $\Omega_a^{32}$ with $a=0.1$ (center) and the limit as $a\to 0$ 
(right).}
\label{fig:mix-hel}
\end{figure} 

\begin{figure}
\begin{center}
\begin{tabular}{ccc}
 \includegraphics[width=.35\linewidth]{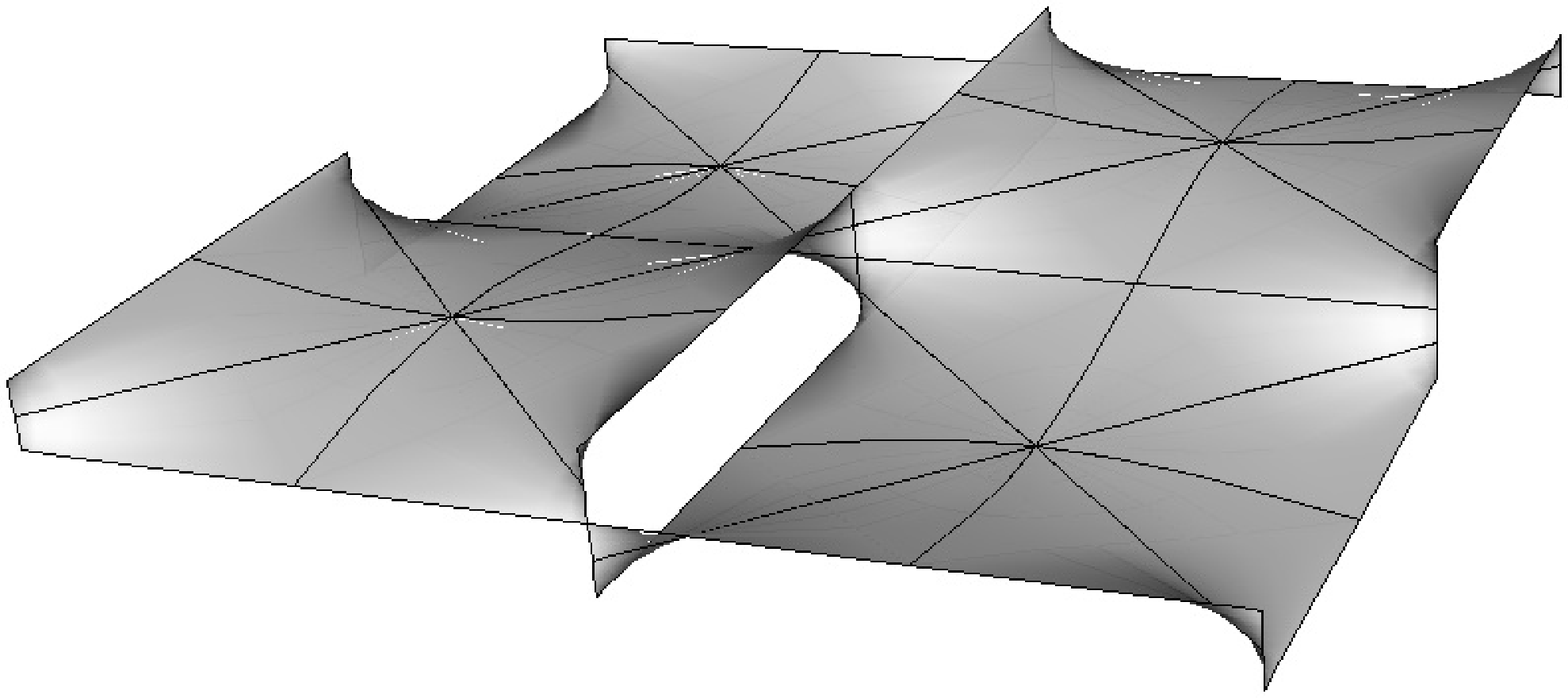} & 
 \includegraphics[width=.35\linewidth]{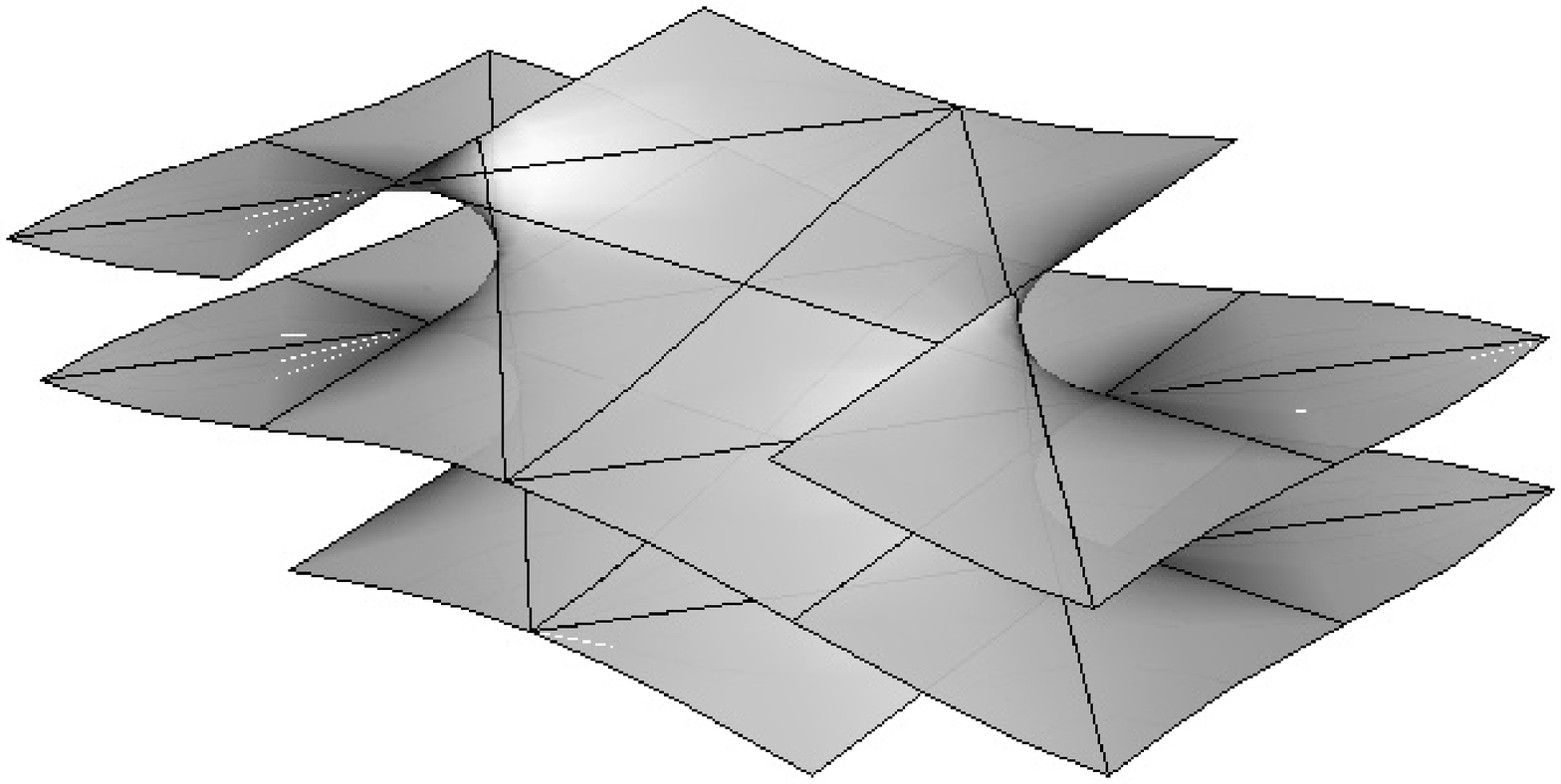} & 
 \includegraphics[width=.25\linewidth]{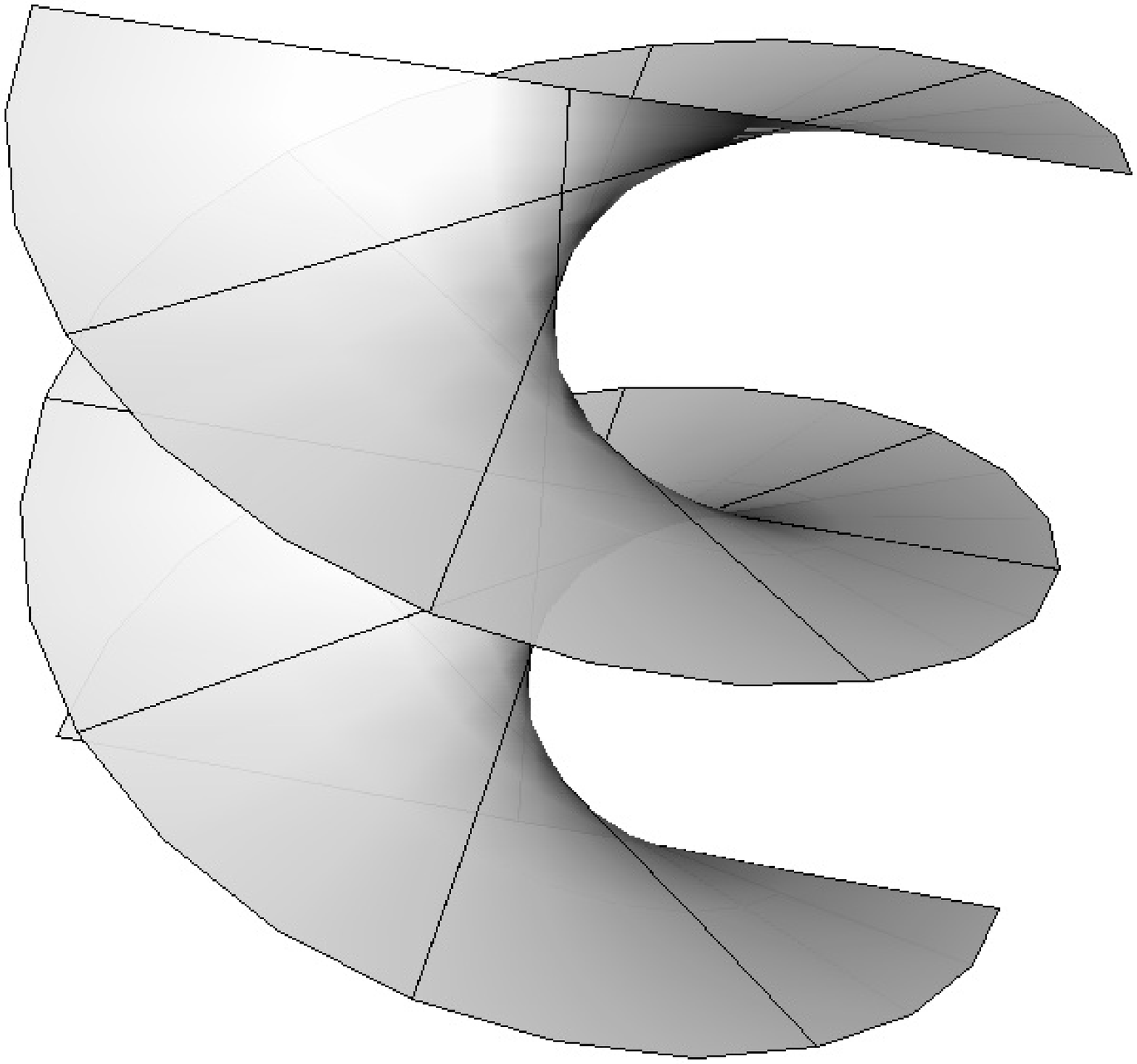} \\
 \end{tabular}
\end{center}
\caption{Schwarz D surface in $\R^3$ with $a=0.1$ (left), 
another view of the Schwarz D surface in $\R^3$ with $a=0.1$ (center) and the helicoid 
in $\R^3$ as a limit as $a\to 0$ with suitable rescaling of Schwarz D surface (right).}
\label{fig:min-hel}
\end{figure} 
\end{remark}

\section{Proof of the embeddedness}
\label{sc:proof-prop}

In this section we present a proof of Proposition~\ref{pr:key-prop}. 

We denote by $\Delta\times\mathcal{L}_C^\text{min}$ the vertical prism over the 
isosceles right triangle $\Delta$ with height $|\mathcal{L}_C^\text{min}|$ 
as in Proportion~\ref{pr:key-prop}. 

First we prepare three lemmas. 

\begin{lemma}
$\Omega_a^\text{max}$ is embedded and contained in 
the closure $\overline{\Delta\times\mathcal{L}_C^\text{min}}$ 
of $\Delta\times\mathcal{L}_C^\text{min}$. 
\end{lemma}

\begin{proof}
The projection of $\Omega_a^\text{max}$ into the $x_1x_2$-plane is represented by 
$$
 \Re\int \left(1+z^2,\,i(1-z^2)\right) i\frac{dz}{w}
=\Re\int \left(i(1+z^2),\,-(1-z^2)\right) \frac{dz}{w},
$$
which is the same as the projection of the Schwarz P minimal surface in $\R^3$ 
(see \eqref{eq:min-r3}). 
So $\Omega_a^\text{max}$ is a graph over $x_1x_2$-plane. 
Since the boundary $\partial\Omega_a^\text{max}$ of $\Omega_a^\text{max}$ is contained in 
$\overline{\Delta\times\mathcal{L}_C^\text{min}}$, $\Omega_a^\text{max}$ itself is contained in 
$\overline{\Delta\times\mathcal{L}_C^\text{min}}$ as well, by the maximum principle. 
Thus $\Omega_a^\text{max}$ is embedded 
and contained in $\overline{\Delta\times\mathcal{L}_C^\text{min}}$. 
\end{proof}

\begin{lemma}\label{lm:convex}
The projection of the singular curve $\gamma_a (s)$ $(0\leq s\leq 2\pi)$
in \eqref{eq:gamma}
into $x_1x_2$-plane is 
a closed convex curve. 
\end{lemma}

\begin{proof}
Let $(x_1(s),x_2(s))$ be the projection of $\gamma_a(s)$ 
into $x_1x_2$-plane. 
It is trivial to see that $(x_1(s),x_2(s))$ is a closed
$C^{\infty}$-regular curve. 
Now we compute the curvature $\kappa_a(s)$ of $(x_1(s),x_2(s))$ and see that 
\begin{equation*}
	\kappa_a(s) = \frac{ \dot{x}_1(s) \ddot{x}_2(s) - \dot{x}_2(s) \ddot{x}_1(s) }
		{ ( \dot{x}_1^2(s) + \dot{x}_2^2(s) )^{3/2}}
		= \sqrt{ \xi_a(s)} >0.
\end{equation*}
So, $(x_1(s),x_2(s))$ is a convex curve.
\end{proof}

\begin{lemma}
$\overline{\Omega_a^\text{max}}\cap\overline{\Omega_a^\text{min}}
=\{\gamma_a(s)\,;\,0\le s\le \pi/4\}$, where $\overline{\Omega_a^\text{max}}$ 
{\rm(}resp. $\overline{\Omega_a^\text{\rm min}}${\rm)} 
 is the closure of $\Omega_a^\text{max}$ 
{\rm(}resp. $\Omega_a^\text{\rm min}${\rm)}. 
\end{lemma}

\begin{proof}
We note that $\tilde f_a(u,v)$ is the midpoint of $\gamma_a(u+v)$ and $\gamma_a(u-v)$.
Therefore, the projection of $\tilde f_a(u,v)$ into $x_1x_2$-plane is the midpoint of 
the projections of $\gamma_a(u+v)$ and $\gamma_a(u-v)$ into $x_1x_2$-plane.
Hence, the projection of $\tilde f_a(u,v)$ into $x_1x_2$-plane is inside the convex curve. 
The claim follows. 
\end{proof}

Thus, to prove Proposition~\ref{pr:key-prop}, it suffices to show that 
$\Omega_a^\text{min}$ is embedded and contained in 
the closure $\overline{\Delta\times\mathcal{L}_C^\text{min}}$ 
of $\Delta\times\mathcal{L}_C^\text{min}$, 
by the above three lemmas. 
To do this, we reparametrize $\gamma_a$ as well as $\tilde f_a$ by their height as follows:

We define a diffeomorphism $\tau :\R\to\R$ by 
$$
\tau (s):=\int_0^s\xi_a(t)dt. 
$$
$\tau (s)$ is the height function (i.e. the $x_0$-component) of $\gamma_a(s)$. 
Using the inverse function $s=s(\tau)$ of $\tau (s)$, we define a parameter change 
$$
\tilde{\gamma}_a(\tau):=\gamma_a(s(\tau)). 
$$
of $\gamma_a(s)$.  
We also define $(\alpha,\beta) :\R^2\to\R^2$ by
$$
(\alpha ,\beta)
=\big(\alpha (u,v), \beta (u,v)\big)
:=\left( \frac{\tau(u+v)+\tau(u-v)}{2} ,\, \frac{\tau(u+v)-\tau(u-v)}{2}\right), 
$$
and 
$$
\check{f}_a(\alpha,\beta)
:=\frac{1}{2}\left(\tilde{\gamma}_a(\alpha +\beta)+\tilde{\gamma}_a(\alpha -\beta)\right).
$$
Since $\tilde{\gamma}_a(\alpha\pm\beta)=\tilde{\gamma}_a(\tau(u\pm v))=\gamma_a(u\pm v)$, 
we see that $\check{f}_a(\alpha,\beta)$ and $\tilde{f}_a(u,v)$ give the same surface 
(see \cite[Proposition 2.2]{CR2}). 
We set 
$$
c_a:=\tau (\pi) =\int_0^\pi\xi_a(t)dt.
$$

\begin{lemma}\label{Lem:20120715_1}
Consider the map $\psi : \R\times (0,\pi)\ni (u,v) \mapsto (\alpha,\beta)\in \R\times (0,c_a)$. 
Then, $\psi$ is a diffeomorphism and the image of the rectangle
$$ 
0 \le u \le \pi/4, \qquad 0 \le v \le \pi/2 
$$ 
is again a rectangle, which is given by 
\begin{equation*}
	0 \le \alpha \le \tau(\pi/4),  \qquad
	0 \le \beta  \le \tau(\pi/2)=c_a/2.
\end{equation*}
\end{lemma}

\begin{proof}
It is easy to see that 
$$
(u,v)=\left( \frac{\tau^{-1}(\alpha + \beta)+\tau^{-1}(\alpha - \beta)}{2} ,\, 
            \frac{\tau^{-1}(\alpha + \beta)-\tau^{-1}(\alpha - \beta)}{2}\right)
$$
gives the inverse function for  $\psi$. 
Since the Jacobian 
$$
\det\frac{\partial (\alpha ,\beta)}{\partial (u,v)}=\xi_a(u+v)\xi_a(u-v)
$$
is always positive, $\psi$ is a diffeomorphism. 

Furthermore, we see that for any $n\in\Z$, 
\begin{align*}
\beta(u,v) &= \frac{1}{2} ( \tau(u+v) - \tau(u-v)) = \frac{1}{2} \int_{u-v}^{u+v} \xi_a(t) dt, \\
\beta(u,n\pi/2) &= \frac{1}{2} \int_{u-n\pi/2}^{u+n\pi/2} \xi_a(t) dt 
= \frac{1}{2} \int_{0}^{n\pi} \xi_a(t) dt =n\tau(\pi/2) \\
&\quad\text{(since $\xi_a(t)$ is periodic with period $\pi/2$)}, \\
\alpha(0,v) &= \frac{1}{2} \Big(\int_{0}^{v} \xi_a(t)dt + \int_{0}^{-v} \xi_a(t)dt \Big) =0
\quad\text{(since $\xi_a(t)$ is even)}, \\
\alpha(\pi/4,v) 
&= \frac{1}{2} \left(\int_{0}^{\pi/4+v} \xi_a(t)dt + \int_{0}^{\pi/4-v} \xi_a(t)dt \right) \\
&= \frac{1}{2} \left( \int_0^{\pi/4} + \int_{\pi/4}^{\pi/4+v} + \int_0^{\pi/4} + \int_{\pi/4}^{\pi/4-v} \right)  \xi_a(t)\,dt\\
& = \int_0^{\pi/4} \xi_a(t)dt=\tau(\pi/4), 
\end{align*}
from which the rest of the claim follows.
\end{proof}

By this lemma, we have that 
$$
\Omega^\text{min}_{a}
=\{\check{f}_a(\alpha,\beta)\in\R^3_1\,;\,0\le \alpha\le\tau(\pi/4),\,0<\beta\le\tau(\pi/2)\}.
$$

\begin{remark}\label{re:height-alpha}
The $x_0$-component of $\check{f}_a(\alpha, \beta)$ is $\alpha$, that is, 
\begin{equation*}
	x_0 \circ \check{f}_a(\alpha,\beta) 
	= \frac{1}{2} \left( \int_{0}^{\alpha+\beta}dt + \int_{0}^{\alpha-\beta} dt \right) 
	= \alpha.
\end{equation*}
\end{remark}

Now we prove the embeddedness of $\Omega^\text{min}_{a}$. 
In fact, we can prove the following stronger lemma. 

\begin{lemma}
$\check{f}_a(\alpha ,\beta)$ $(\alpha\in\R$, $\beta\in (0,c_a))$ is embedded. 
\end{lemma}

\begin{proof}
Suppose that $\check{f}_a(\alpha ,\beta)=\check{f}_a(\alpha' ,\beta')$, where 
$\alpha, \alpha'\in\R$ and $\beta,\beta'\in (0,c_a)$. 
Then by Remark~\ref{re:height-alpha}, we have $\alpha=\alpha'$. 
Suppose now that $\beta<\beta'$. 
Let $\pi_0:\R^3_1\ni (x_0,x_1,x_2)\mapsto (x_1,x_2)\in\R^2$ be the projection. 
By Lemma~\ref{lm:convex}, $\pi_0\circ\tilde{\gamma}_a(\tau)$ is a closed convex curve. 
Since $\pi_0\circ\gamma_a(s)$ is $2\pi$-periodic, we see that 
$\pi_0\circ\tilde{\gamma}_a(\tau)$ is $2c_a$-periodic. 
Since $0<\beta<\beta'<c_a$, 
$$
\pi_0\circ\tilde{\gamma}_a(\alpha-\beta'),\quad
\pi_0\circ\tilde{\gamma}_a(\alpha-\beta),\quad
\pi_0\circ\tilde{\gamma}_a(\alpha),\quad
\pi_0\circ\tilde{\gamma}_a(\alpha+\beta),\quad
\pi_0\circ\tilde{\gamma}_a(\alpha+\beta')
$$
lie on the curve $\pi_0\circ\tilde{\gamma}_a(\tau)$ in this order. 
The assumption that $\check{f}_a(\alpha ,\beta)=\check{f}_a(\alpha ,\beta')$ implies that 
the midpoint of $\pi_0\circ\tilde{\gamma}_a(\alpha-\beta')$ and 
$\pi_0\circ\tilde{\gamma}_a(\alpha+\beta')$ is equal to the midpoint of 
$\pi_0\circ\tilde{\gamma}_a(\alpha-\beta)$ and 
$\pi_0\circ\tilde{\gamma}_a(\alpha+\beta)$, which is a contradiction by the convexity of 
$\pi_0\circ\tilde{\gamma}_a(\tau)$. 
So, $\beta \ge \beta'$. In a similar way we can conclude $\beta' \ge \beta$, hence $\beta = \beta'$. This finishes the proof. 
\end{proof}

Hence proving the following lemma completes the proof of Proposition~\ref{pr:key-prop}. 
 
\begin{lemma}
$\Omega^\text{\rm min}_{a}$ is contained in 
the closure $\overline{\Delta\times\mathcal{L}_C^\text{\rm min}}$ 
of $\Delta\times\mathcal{L}_C^\text{\rm min}$. 
\end{lemma}

\begin{proof}
Direct computations show that 
$$
\frac{\partial\check{f}_a}{\partial\alpha}
= (1,-2\cos u\cos v, -2\sin u\cos v),\qquad
\frac{\partial\check{f}_a}{\partial\beta}
= (1,2\sin u\sin v, -2\cos u\sin v).
$$
Because $u\in (0,\pi/4)$, $v\in (0,\pi/2)$, we have 
$$
-2\cos u\cos v <0, \quad
-2\sin u\cos v <0, \quad
2\sin u\sin v >0, \quad
-2\cos u\sin v <0.
$$  
Since the boundary $\partial\Omega^\text{min}_{a}$ of $\Omega^\text{min}_{a}$ 
consists of the three straight line segments
\begin{align*}
\mathcal{L}_A^\text{min}&:=\{\check{f}_a(0,\beta)\in\R^3_1\,;\,0<\beta\le c_a/2\},\\
\mathcal{L}_B^\text{min}&:=\{\check{f}_a(\tau(\pi/4),\beta)\in\R^3_1\,;\,0<\beta\le c_a/2\},\\
\mathcal{L}_C^\text{min}&:=\{\check{f}_a(\alpha,c_a/2)\in\R^3_1\,;\,0\le \alpha\le\tau (\pi/4)\},
\end{align*}
and the singular curve $\tilde{\gamma}_a(\tau)$ ($0\le \tau \le \tau(\pi/4)$), 
and all of them are contained in $\overline{\Delta\times\mathcal{L}_C^\text{min}}$, the claim follows. 
\end{proof}


\end{document}